\documentclass{article}

\usepackage{epsfig,indentfirst,latexsym,bm,amsmath,amssymb,amsfonts,lscape,rawfonts,eufrak}
\usepackage{verbatim}
\usepackage{graphicx}
\usepackage{subfigure}
\usepackage{epsfig}
\usepackage{longtable}
\usepackage{amsthm,color}

\usepackage{cite}

\def\bf{{\mathbf f}}
\def\bk{{\mathbf k}}
\def\bp{{\mathbf p}}
\def\bd{{\mathbf d}}
\def\be{{\mathbf e}}
\def\bu{{\mathbf u}}
\def\bv{{\mathbf v}}
\def\bq{{\mathbf q}}
\def\bA{{\mathbf A}}
\def\bW{{\mathbf W}}

\def\R{{{\mathbb R}}}
\def\N{{{\mathbb N}}}
\def\Z{{{\mathbb Z}}}
\newcommand{\LRs}{L_2(\mathbb{R}^{\rm dim})}
\newcommand{\LR}{L_2(\mathbb{R})}
\def\wt{\widetilde}

\def\gbo{{\boldsymbol \omega}}

\newtheorem{definition}{Definition}[section]
\newtheorem{theorem}{Theorem}[section]
\newtheorem{proposition}{Proposition}[section]
\newtheorem{lemma}{Lemma}[section]
\newtheorem{corollary}{Corollary}[section]
\newtheorem{example}{Example}[section]

\newtheorem{remark}{Remark}[section]

\numberwithin{equation}{section}

\begin{document}

\title{Image Restoration: A General Wavelet Frame Based Model and Its Asymptotic Analysis}
\author{Bin Dong${}^{1}$,\; Zuowei Shen${}^2$ \; and Peichu Xie ${}^2$\\
\\
${}^{1}${\it {\small  Beijing International Center for Mathematical Research (BICMR)}}\\
{\it {\small Peking University, Beijing, China, 100871}}\\
{\small dongbin@math.pku.edu.cn (corresponding author)}\\
\\
${}^2${\it {\small Department of Mathematics, National University of Singapore}}\\
{\small {\it 10 Lower Kent Ridge Road, Singapore, 119076 }}\\
{\small matzuows@nus.edu.sg (Zuowei Shen),  \; xie@u.nus.edu (Peichu Xie)} }

\footnotetext[1]{Bin Dong is supported in part by the Thousand Talents Plan of China.}
\footnotetext[2]{Zuowei Shen is supported by the Tan Chin Tuan Centennial Professorship at National University of Singapore.}

\maketitle
%\tableofcontents
%\pagebreak

\begin{abstract}
Image restoration is one of the most important areas in imaging science. Mathematical tools have been widely used in image restoration, where wavelet frame based approach is one of the successful examples. In this paper, we introduce a generic wavelet frame based image restoration model, called the ``general model", which includes most of the existing wavelet frame based models as special cases. Moreover, the general model also includes examples that are new to the literature. Motivated by our earlier studies \cite{CDOS2011,DJS2013,CDS2014}, We provide an asymptotic analysis of the general model as image resolution goes to infinity, which establishes a connection between the general model in discrete setting and a new variatonal model in continuum setting. The variational model also includes some of the existing variational models as special cases, such as the total generalized variational model proposed by \cite{bredies2010total}. In the end, we introduce an algorithm solving the general model and present one numerical simulation as an example.
\end{abstract}

\section{Introduction}

Image restoration, including image denoising, deblurring, inpainting, medical imaging, etc., is one of the most important areas in imaging science. Image restoration problems can be formulated as the following linear inverse problem
\begin{equation}\label{LinearInverseProblem}
\mathbf f=\mathbf A \mathbf u+\mathbf \varepsilon,
\end{equation}
where the matrix $\mathbf{A}$ is some linear operator (not invertible in general) and $\mathbf{\varepsilon}$ denotes a perturbation caused by the additive noise in the observed image, which is typically assumed to be white Gaussian noise. As convention, we regard an image as a discrete function $\mathbf u$ defined on a regular grid $\mathbf{O}\subset (h\mathbb Z)^2$ (where $h$ indicates the size of each pixel): $\mathbf u:\mathbf O\to\mathbb{R}$.

Different image restoration problem corresponds to a different type of $\mathbf{A}$ in \eqref{LinearInverseProblem}. For example, $\mathbf A$ is the identity operator for image denoising, a restriction operator for image inpainting, a convolution operator for image deblurring, a partial collection of line integrations for CT imaging, a partial Fourier transform for MR Imaging, etc. The problem \eqref{LinearInverseProblem} is usually ill-posed, which makes solving \eqref{LinearInverseProblem} non-trivial. A naive inversion of $\mathbf{A}$ may result in a recovered image with amplified noise and smeared-out edges. A good image restoration method should be capable of smoothing the image so that noise is suppressed to the greatest extend, while at the same time, restoring or preserving important image features such as edges, ridges, corners, etc.

Most of the existing image restoration methods are transformation based. A good transformation for image restoration should be capable of capturing both global patterns and local features of images. The global patterns are smooth image components that provide a global view of images, while the local features are sharp image components that characterize local singularities of images. Wavelet frame transform is one of the successful examples. Wavelet frames represent images as an addition of global patterns, i.e. smooth image components, and local features, i.e. image singularities. In wavelet frame domain, global patterns are represented by densely distributed coefficients obtained from low-pass filtering, while local features are represented by sparse coefficients obtained from high-pass filtering. Therefore, wavelet frames can effectively separate smooth image components and image features, which is the key to their success in image restoration. Thanks to redundancy, wavelet frame systems have enough flexibility to better balance between smoothness and sparsity than (bi)orthogonal wavelets so that artifacts generated by the Gibbs phenomenon can be further reduced, which in turn leads to better image reconstruction. In addition to providing sparse approximation to local image features, the large coefficients from high-pass filtering can also be used to accurately detect the locations and estimate the types of image singularities. In other words, these coefficients also provide reliable analysis and classifications of local image features in the transform domain.

There are many different wavelet or wavelet frame based image restoration models proposed in the literature including the synthesis based approach \cite{daubechies2007iteratively,fadili2005sparse,fadili2009inpainting,figueiredo2003algorithm,figueiredo2005bound}, the analysis based approach \cite{cai2009split,elad2005simultaneous,starck2005image}, and the balanced approach \cite{chan2003wavelet,CCSS,cai2008simultaneous}. For images that are better represented by a composition of two layers which can each be sparsely approximated by two different frame systems, two-system models were proposed in \cite{cai2009split,elad2005simultaneous, starck2005image,dong2011wavelet,CDOS2011}. Although all these models are different in form from each other, they all share the same modeling philosophy, i.e. to penalize the $\ell_1$-norm (or more generally, any sparsity promoting norms) of the sparse coefficients in wavelet frame domain. This is because wavelet frame systems can sparsely approximate local features of piecewise smooth functions such as images.

In this paper, we study a generic wavelet frame based image restoration model which includes most of the aforementioned models as special cases. This model shall be referred to as the ``general model" for image restoration. Moreover, the general model also includes some models that are new to the literature. Now, we present the general model for wavelet frame based image restoration as follows:
\begin{equation}\label{uni}
\inf_{\mathbf u,\mathbf v}\left\{a\|\mathbf{W'}\mathbf{u}-\mathbf{v}\|^p_{\ell_p(\mathbf O)}+b\|\mathbf W''\mathbf{v}\|^q_{\ell_q(\mathbf O)}+\frac{1}{2}\|\mathbf A \mathbf u-\mathbf f\|_{\ell_2(\mathbf O)}^2\right\},
\end{equation}
where $\mathbf W'$ and $\mathbf W''$ are wavelet frame transforms associated to two wavelet frame systems, and $1\leq p,q\leq 2$. Here, $\mathbf u$ is the image to be recovered, and $\mathbf v$ lives in the transform domain of $\mathbf W'$ which is essentially a vector field. The wavelet frame system corresponding to the transform $\mathbf W''$ consists of subsystems that are applied to each of the components of $\mathbf v$. For clarity of the presentation, details of the definition of \eqref{uni} will be postponed to a later section.

Now, we observe that the general model \eqref{uni} indeed takes many existing wavelet frame based models as special cases.

\medskip{\noindent\textbf{Case A:}} Let $\mathbf W'=\mathbf W$ be a certain wavelet frame transform, and $\mathbf W''=\emph{\emph{Id}}$; and choose $p=2,\ q=1$. By fixing $\mathbf u\equiv\mathbf W^T\mathbf v$, the general model \eqref{uni} becomes the \emph{Balanced Model} of \cite{chan2003wavelet,CCSS,cai2008simultaneous}:
\begin{equation}\label{balanced_discrete}
\inf_{\mathbf v}\left\{a\|(\emph{Id}-\mathbf W\mathbf W^T)\mathbf v\|^2_{\ell_2(\mathbf O)}+b\|\mathbf v\|_{\ell_1(\mathbf O)}+\frac{1}{2}\|\mathbf A \mathbf W^T\mathbf v-\mathbf f\|_{\ell_2(\mathbf O)}^2\right\}
\end{equation}
If we further enforce the condition $a=0$ in model (\ref{balanced_discrete}), then we obtain the \emph{Synthesis Model} \cite{daubechies2007iteratively,fadili2005sparse,fadili2009inpainting,figueiredo2003algorithm,figueiredo2005bound}:
\begin{equation}\label{synthesis_discrete}
\inf_{\mathbf v}\left\{b\|\mathbf{v}\|_{\ell_1(\mathbf O)}+\frac{1}{2}\|\mathbf A \mathbf W^T\mathbf v-\mathbf f\|_{\ell_2(\mathbf O)}^2\right\}
\end{equation}
If we formally set $a=\mathcal 1$ in (\ref{balanced_discrete}), or more strictly, set $\mathbf v= \mathbf 0$ directly in (\ref{uni}), we obtain the following \emph{Analysis Model} \cite{cai2009split,elad2005simultaneous,starck2005image}:
\begin{equation}\label{analysis_discrete}
\inf_{\mathbf u}\left\{a\|\mathbf{W}\mathbf{u}\|_{\ell_1(\mathbf O)}+\frac{1}{2}\|\mathbf A \mathbf u-\mathbf f\|_{\ell_2(\mathbf O)}^2\right\}
\end{equation}

\medskip{\noindent\textbf{Case B:}} Let $\mathbf W'=\mathbf W''=\mathbf W$, $\mathbf v=\mathbf W\mathbf u_2$, $\mathbf u=\mathbf u_1+\mathbf u_2$, and $p=q=1$. The general model (\ref{uni}) becomes the two-layers \emph{Wavelet-Packet Model} of \cite{CDOS2011}:
\begin{equation}\label{packet_discrete}
\inf_{\mathbf u_1,\mathbf u_2}\left\{a\|\mathbf W\mathbf u_1\|_{\ell_1(\mathbf O)}+b\|\mathbf W^2\mathbf u_2\|_{\ell_1(\mathbf O)}+\frac{1}{2}\|\mathbf A (\mathbf u_1+\mathbf u_2)-\mathbf f\|_{\ell_2(\mathbf O)}^2\right\}.
\end{equation}

\medskip{\noindent\textbf{Case C:}} Let $p=q=1$. The general model (\ref{uni}) becomes the more general two-layers model proposed in \cite{CDOS2011}:
\begin{equation}\label{twolayers_discrete}
\inf_{\mathbf u,\mathbf v}\left\{a\|\mathbf{W'}\mathbf{u}-\mathbf{v}\|_{\ell_1(\mathbf O)}+b\|\mathbf W''\mathbf{v}\|_{\ell_1(\mathbf O)}+\frac{1}{2}\|\mathbf A \mathbf u-\mathbf f\|_{\ell_2(\mathbf O)}^2\right\}.
\end{equation}

The general model \eqref{uni} also includes the following new model as its special case.

\medskip{\noindent\textbf{Case D (New):}} Let $p=1$ and $q=2$. The general model (\ref{uni}) becomes the following model:
\begin{equation}\label{newmodel_discrete}
\inf_{\mathbf u,\mathbf v}\left\{a\|\mathbf{W'}\mathbf{u}-\mathbf{v}\|_{\ell_1(\mathbf O)}+b\|\mathbf W''\mathbf{v}\|^2_{\ell_2(\mathbf O)}+\frac{1}{2}\|\mathbf A \mathbf u-\mathbf f\|_{\ell_2(\mathbf O)}^2\right\}.
\end{equation}
When model \eqref{newmodel_discrete} is used, the image to be recovered is understood as having two layers: one layer contains sharp image features while the other layer consists of smooth image components. To see this, we let $\mathbf v=\mathbf W'\mathbf u_2$, $\mathbf u=\mathbf u_1+\mathbf u_2$. Then \eqref{newmodel_discrete} becomes
\begin{equation}\label{newmodel_discrete:u1u2}
\inf_{\mathbf u_1,\mathbf u_2}\left\{a\|\mathbf W'\mathbf u_1\|_{\ell_1(\mathbf O)}+b\|\mathbf W''\mathbf W'\mathbf u_2\|^2_{\ell_2(\mathbf O)}+\frac{1}{2}\|\mathbf A (\mathbf u_1+\mathbf u_2)-\mathbf f\|_{\ell_2(\mathbf O)}^2\right\}.
\end{equation}
The penalization of the $\ell_1$-norm of $\mathbf W'\mathbf u_1$ ensures sharp image features are well captured by $\mathbf u_1$, while the penalization of the $\ell_2$-norm of $\mathbf W''\mathbf W'\mathbf u_2$ ensures the smooth image components are well captured by $\mathbf u_2$. Note that we present model \eqref{newmodel_discrete:u1u2} to show that \eqref{newmodel_discrete} implicitly assumes that the image to be recovered contains two layers. They are not equivalent in general since $\mathbf v$ in \eqref{newmodel_discrete} does not have to be in the range of $\mathbf W'$.

Notably, the model \eqref{newmodel_discrete} is related to the newly proposed \emph{piecewise smooth image restoration model} of \cite{CDS2014}. We first recall the piecewise smooth model of \cite{CDS2014} as follows
\begin{equation}\label{Energy:PS:CDS}
\inf_{\mathbf{u},\ \mathbf{\Lambda}\subset\mathbf O}\ a\left\|\left[\mathbf{W}\mathbf{u}\right]_{\mathbf{\Lambda}}\right\|_{\ell_1(\mathbf O)}+b\left\|\left[\mathbf{W}\mathbf{u}\right]_{\mathbf{\Lambda}^c}\right\|_{\ell_2(\mathbf O)}^2+\frac12\|\mathbf{Au}-\mathbf{f}\|_{\ell_2(\mathbf O)}^2,
\end{equation}
where $\mathbf \Lambda$ is a sub-index set of $\mathbf O$ that indicates the locations of the image singularities, and $\left[\mathbf{W}\mathbf{u}\right]_{\mathbf{\Lambda}}$ (resp. $\left[\mathbf{W}\mathbf{u}\right]_{\mathbf{\Lambda^c}}$) denotes the restricted coefficients on set $\mathbf\Lambda$ (resp. $\mathbf \Lambda^c$). The image recovered by the piecewise smooth model \eqref{Energy:PS:CDS} can be written as $\mathbf u=\mathbf u_1+\mathbf u_2$ where
\begin{equation*}
\mathbf u_1=
\begin{cases}
[\mathbf u]_{\mathbf{\Lambda}} & \mbox{on } \mathbf\Lambda\cr
0&\mbox{elsewhere}
\end{cases}
\quad\mbox{and}\quad \mathbf u_2=
\begin{cases}
[\mathbf u]_{\mathbf{\Lambda}^c} & \mbox{on } \mathbf\Lambda^c\cr
0&\mbox{elsewhere.}
\end{cases}
\end{equation*}
Therefore, the piecewise smooth image restoration model \eqref{Energy:PS:CDS} assumes the image to be recovered consists of two layers where one contains sharp image features and the other contains smooth image components.

Comparing model \eqref{newmodel_discrete:u1u2} with the piecewise smooth model \eqref{Energy:PS:CDS}, we can see that both models assume the images to be recovered can be decomposed into an addition of sharp image features $\mathbf u_1$ and smooth image components $\mathbf u_2$ via the penalization of the $\ell_1$-norm of the wavelet frame coefficients of $\mathbf u_1$ and the $\ell_2$-norm of the wavelet frame coefficients of $\mathbf u_2$. However, the difference between them is that $\mathbf u_1$ of the piecewise smooth model contains only sharp image features while $\mathbf u_1$ of the model \eqref{newmodel_discrete:u1u2} contains both sharp and some smooth image components. In other words, the decomposition $\mathbf u=\mathbf u_1+\mathbf u_2$ of the piecewise smooth model is non-overlapping, while that of the model \eqref{newmodel_discrete:u1u2} has some overlaps. It is not clear at this point whether such overlapping will lead to better image restoration results or not. However, model \eqref{newmodel_discrete:u1u2} (as well as model \eqref{newmodel_discrete}) is convex while the piecewise smooth model \eqref{Energy:PS:CDS} is nonconvex. Therefore, one may expect better behavior and theoretical support for the numerical algorithms solving \eqref{newmodel_discrete:u1u2} (and \eqref{newmodel_discrete} as well). To properly compare the two models and their associated algorithms, we need to conduct comprehensive numerical studies. However, we shall omit these numerical studies in this paper since our focus is to provide a theoretical study of the general model \eqref{uni}. Nonetheless, numerical experiments in \cite{CDS2014} on the piecewise smooth model \eqref{Energy:PS:CDS} showed the advantage of the modeling philosophy that is adopted by both \eqref{Energy:PS:CDS} and \eqref{newmodel_discrete}, i.e. modeling images as a summation of one image layer encoding sharp image features and another layer encoding smooth image components, and penalizing the $\ell_1$-norm and $\ell_2$-norm of them in transform domain respectively.

\subsection{Analyzing Model \eqref{uni}}

The main objective, as well as contribution, of this paper is to provide an asymptotic analysis of the general model \eqref{uni} as image resolution goes to infinity. This work is motivated by earlier studies of \cite{CDOS2011,DJS2013,CDS2014}, where it was shown that wavelet frame transforms are discretization of differential operators in both variational and PDE frameworks, and such discretization is superior to some of the traditional finite difference schemes for image restoration. In particular, fundamental connection of wavelet frame based approach to total variation model \cite{ROF} was established in \cite{CDOS2011}, to the Mumford-Shah model \cite{mumford1989optimal} was established in \cite{CDS2014} and to nonlinear evolution PDEs in \cite{DJS2013}. This new understanding essentially merged the two seemingly unrelated areas: wavelet frame base approach and PDE based approach. It also gave birth to many innovative and more effective image restoration models and algorithms. Therefore, an asymptotic analysis of the general model \eqref{uni} is important to the understanding of the model, as well as the corresponding variational model.

In \cite{CDOS2011}, asymptotic analysis of the wavelet frame based analysis model \eqref{analysis_discrete} was provided. Asymptotic analysis of the piecewise smooth model \eqref{Energy:PS:CDS} with a fixed $\mathbf\Lambda$ was given in \cite{CDS2014}. However, asymptotic analysis of many other wavelet frame based models proposed in the literature, such as the examples in Case A-D we presented in the previous subsection, is still missing. In this paper, we give a unified analysis of all these wavelet frame based models by providing an asymptotic analysis of the general model \eqref{uni}.

In model \eqref{uni}, we view images as data samples of functions at a given resolution. The discrete wavelet frame coefficients are obtained by applying wavelet frame filters to the given image data. Since the operation of high-pass filtering in the wavelet frame transform can be regarded as applying a certain finite difference operator on the image, one can easily show using Taylor's expansion that when images are sampled from functions that are smooth enough, wavelet frame transforms indeed approximate differential operators if each of the wavelet frame band is properly weighted. This motivates us that there is a certain variational model to which the general model \eqref{uni} approximates. However, we need to justify such approximation in more general function spaces than smooth function spaces since images are by no means smooth. This requires more sophisticated analysis than simple Taylor's expansion.

The analysis used in this paper is based on what was used in \cite{CDOS2011,CDS2014}. Let $\psi_{n,\bk}$ be a wavelet frame function and $\phi_{n,\bk}$ the corresponding refinable function at scale $n\in\Z$ and location $\bk\in\mathbf O$. An image $\mathbf u$ is understood as a discrete sample of the associated function $u$ via the inner product $\mathbf u[\bk]=\delta_n\langle u,\phi_{n,\bk}\rangle$ where $\delta_n$ is a constant depending on $n$. When discrete wavelet frame transform is applied on $\mathbf u$, the transform corresponding to the element $\psi_{n-1,\bk}$ produces a coefficient proportional to $\langle u,\psi_{n-1,\bk}\rangle$. One key observation that is crucial to the analysis of \cite{CDOS2011,CDS2014} is that there exists a function $\varphi$ associated to $\psi$ with non-zero integration and enough smoothness, such that $\langle u,\psi_{n-1,\bk}\rangle$ is proportional to $\langle D u,\varphi_{n-1,\bk}\rangle$, where $D$ is a differential operator depending on the property of the wavelet frame function $\psi$. In other words, the wavelet frame coefficient associated to $\psi_{n-1,\bk}$ can be understood as a sampling of $D u$ via $\widetilde\delta_n\langle D u,\varphi_{n-1,\bk}\rangle$ with $\widetilde\delta_n$ a constant depending on $n$.

Based on the aforementioned observations, we are able to find the variational model corresponding to the general model \eqref{uni}. We will show that the objective function of an equivalent form of the general model \eqref{uni} $\Gamma$-converges (see e.g. \cite{dal1993introduction}) to the energy functional of the variational model as image resolution goes to infinity. Through the $\Gamma$-convergence, connections of the approximate minimizers of the general model to those of the variational model are also established. A summary of our main findings is given in the next subsection.

\subsection{Main Results}

We assume all functions/images we consider are defined on the open unit square $\Omega:=(0,1)^2\subset\R^2$. Let $\mathbf O_n\subset\Omega$ be a $2^n\times 2^n$ Cartesian grid on $\bar\Omega$ with $n\in\N$ indicating the resolution of the grid. Let $\mathbf K_n \subset \mathbf O_n$ be an appropriate index set on which wavelet transforms are well-defined. Let $\mathbf u_n$ be a real-valued array defined on $\mathbf K_n$, i.e. $\mathbf u_n\in\R^{|\mathbf K_n|}$, and $\mathbf v_n$ be a vector-valued array on $\mathbf K_n$ with $J$ components, i.e. $\mathbf v_n\in\R^{J\cdot|\mathbf K_n|}$. Denote $\mathbf W'_n: \R^{|\mathbf K_n|}\mapsto\R^{J|\mathbf K_n|}$ and $\mathbf W''_n:\R^{J|\mathbf K_n|}\mapsto\R^{J^2|\mathbf K_n|}$ be wavelet frame transforms with each band weighted by a certain scalar depending on $n$. Details of these definitions can be found in Section \ref{S:Preliminaries}.

We start with a more precise definition of the general model \eqref{uni}.
\begin{definition}
At a given resolution $n\in\N$, rewrite the general model \eqref{uni} as the following optimization problem:
\begin{equation}\label{main}
\inf_{\mathbf u_n, \mathbf v_n} F_n(\mathbf u_n, \mathbf v_n)
\end{equation}
where
\begin{equation*}
F_n(\mathbf u_n, \mathbf v_n):=\nu_1\|\mathbf W'_n\mathbf{u}_n-\mathbf{v}_n\|^p_{\ell_p(\mathbf{K}_n;,\ell_2)}+\nu_2\|\mathbf W''_n\mathbf{v}_n\|^q_{\ell_q(\mathbf{K}_n;,\ell_2)}+\frac{1}{2}\|\mathbf A_n\mathbf{u}_n-\mathbf{f}_n\|_{\ell_2(\mathbf{K}_n)}^2
\end{equation*}
and the norm of the $m$-vector-valued arrays are defined as follows
\begin{equation*}
\left\|\left(\mathbf f_1,\cdots,\mathbf f_{m}\right)\right\|_{\ell_p(\mathbf K_n;,\ell_q)}:=\left(2^{-2n}\sum_{\bk\in\mathbf{K}_n} \left(\sum_{i=1}^{m} |\mathbf f_i[\bk]|^q\right)^{p/q}\right)^{1/p},
\end{equation*}
with $m=J$ for the first term of $F_n$ and $m=J^2$ for the second term of $F_n$.
\end{definition}

Let the operators $\mathbf T_n: L_2(\Omega)\to\mathbb R^{|\mathbf K_n|}$ and $\mathbf S_n: L_2(\Omega;\mathbb R^J)\to\mathbb R^{J\cdot |\mathbf K_n|}$ be sampling operators (see \eqref{D:Tn} and \eqref{D:Sn} for details). We define the functional $E_n(u,v)$ based on the objective function $F_n(\mathbf u_n,\mathbf v_n)$:
\begin{equation*}
E_n(u,v):=F_n(\mathbf T_n u, \mathbf S_n v)\qquad\mbox{with } u\in L_2(\Omega) \mbox{ and } v\in L_2(\Omega;\mathbb R^J).
\end{equation*}
The relation between the problem $\inf_{\mathbf u,\mathbf v} F_n(\mathbf u_n,\mathbf v_n)$ and $\inf_{u,v} E_n(u,v)$ for a fixed $n$ will be given by Proposition \ref{E_nF_n} which states that $\inf_{\mathbf u_n,\mathbf v_n} F_n(\mathbf u_n,\mathbf v_n)=\inf_{u,v} E_n(u,v)$; and for any given minimizer $(\mathbf u_n^\star,\mathbf v_n^\star)$ of $F_n$, one can find $(u_n^\star, v_n^\star)$ that is a minimizer of $E_n$, and vice versa.

\begin{definition}\label{Def:D':D''}
Let $\mathbf{D}'=(D_1',\ldots,D_J')$ be a general differential operator with $\mathbf D':W_s^p(\Omega)\rightarrow W_{s-|\mathbf D'|}^p(\Omega;\mathbb{R}^J)$ for $s>|\mathbf D'|:=\max_j |D_j'|$. Given $u\in W_s^p(\Omega)$, $\mathbf D'u:=(D'_1u,\ldots, D'_Ju)$. One example of $\mathbf D'$ is $\mathbf D'=\nabla$ with $\nabla: W_s^p(\Omega)\rightarrow W_{s-1}^p(\Omega;\mathbb{R}^2)$.

Given $\mathbf D'$, let $\mathbf{D}''=(\mathbf D',\ldots,\mathbf D')$ with $\mathbf D'':W_s^p(\Omega; \R^J)\rightarrow W_{s-|\mathbf D''|}^p(\Omega;\mathbb{R}^{J^2})$ for $s>|\mathbf D''|=|\mathbf D'|$. Given $v\in W_s^p(\Omega; \R^J)$, $\mathbf D''v:=(D'v_1,\ldots,D'v_J)$. For example, when $\mathbf D'=\nabla$, we have $\mathbf D''=(\frac{\partial}{\partial x_1}, \frac{\partial}{\partial x_2}, \frac{\partial}{\partial x_1}, \frac{\partial}{\partial x_2})$ with $\mathbf D'': W_s^p(\Omega,\R^2)\rightarrow W_{s-1}^p(\Omega;\mathbb{R}^4)$.
\end{definition}

\begin{remark}
The differential operator $\mathbf D''$ in Definition \ref{Def:D':D''} is formed by stacking $J$ copies of $\mathbf D'$. Note that we can make $\mathbf D''$ entirely general, i.e. $\mathbf D''=(D_{ij}'')_{1\le i,j\le J}$. The proof of our main theorem can be modified to facilitate such generalization. We only need to adjust the weights at each band of $\mathbf W''_n$ properly. However, for better readability and clarity, we shall focus on the choice of $\mathbf D''$ in Definition \ref{Def:D':D''}.
\end{remark}

We discovered that the corresponding variational model to the discrete model $E_n(u,v)$ takes the following form:
\begin{equation}\label{E}
E(u,v):=\nu_1\left\|\mathbf D'u-v\right\|^p_{L_p(\Omega;\ell_2)}+\nu_2\left\|\mathbf D''v\right\|^q_{L_q(\Omega;\ell_2)}+\frac{1}{2}\left\|Au-f\right\|_{L_2(\Omega)}^2,
\end{equation}
where the norm of the $m$-vector-valued functions are defined as follows
\begin{equation*}
\left\|\left(f_1,\cdots, f_{m}\right)\right\|_{L_p(\Omega;,\ell_q)}:=\left(\int_\Omega \left(\sum_{i=1}^{m} |f_i[\mathbf x]|^q\right)^{p/q} \rm d\mathbf x\right)^{1/p},
\end{equation*}
with $m=J$ for the first term of $E$ and $m=J^2$ for the second term of $E$.

Our main result reads as follows:

{\noindent\textbf{Theorem \ref{main_1}.}} For any given differential operators $\mathbf D'$ and $\mathbf D''$ given by Definition \ref{Def:D':D''} with order $s>0$, one can always select the wavelet frame transforms $\mathbf W_n'$ and $\mathbf W_n''$ with each wavelet frame bands properly weighted, such that $E_n$ $\Gamma$-converges to $E$ under the topology of $W_{2s}^p(\Omega)\times W_{s}^q(\Omega;\mathbb R^J)$.

Based on the $\Gamma$-convergence of Theorem \ref{main_1}, we have the following result that describes the relation between the ($\epsilon$-optimal) solutions of $E_n$ and those of $E$:

{\noindent\textbf{Corollary \ref{cluster_point}}.} If the sequence of the ($\epsilon$-optimal) solutions of $E_n$ has a cluster point $(u^*,v^*)$, then this cluster point $(u^*,v^*)$ is an ($\epsilon$-optimal) solution of $E$.

We finally note that the variational model \eqref{E} is closely related to the total generalized variational (TGV) model of \cite{bredies2010total} if the infinums in \eqref{E} are successively enforced on $u$ and $v$. In particular, we have
\begin{eqnarray*}
\inf_{v}E(u,v)&=&\inf_{v}\left\{\nu_1\left\|\nabla u-v\right\|_{L_1(\Omega;\ell_2)}+\nu_2\left\|\nabla v\right\|_{L_1(\Omega;\ell_2)}\right\}+\frac{1}{2}\left\|Au-f\right\|_{L_2(\Omega)}^2\cr
&=&TGV_{\nu_2,\nu_1}(u)+\frac{1}{2}\left\|Au-f\right\|_{L_2(\Omega)}^2,
\end{eqnarray*}
where
\begin{eqnarray*}
TGV_{\nu_2,\nu_1}(u)=\inf\left\{\left.\int u[\nabla^2 w]\right|w\in C^{\mathcal{1}}(\Omega;\mathbb{R}^{J^2}),\ \|w\|_{L_\mathcal{1}(\Omega;\ell_2)}\leq \nu_2,\ \|\nabla w\|_{L_\mathcal{1}(\Omega;\ell_2)}\leq \nu_1\right\}.
\end{eqnarray*}

\subsection{Organization of the Paper} In Section \ref{S:Preliminaries}, we start with a review of wavelet frames followed by an introduction of basic notation and properties that will be needed in our analysis. Our main results are presented in Section \ref{S:Theory}, where the proof of the main theorem is given based on two technical lemmas that are proved later in Section \ref{SubS:Lemma:PW} and Section \ref{SubS:Lemma:EqCont} respectively. In Section \ref{S:Algorithm}, we propose an algorithm solving the general model. We also present one numerical simulation as an example.

\section{Preliminaries}\label{S:Preliminaries}

\subsection{Wavelet Frames}\label{SubS:wavelet:frames}

In this section, we briefly introduce the concept of wavelet frames. The interested readers should consult \cite{ron1997affine,ron1997affineII,Dau,Daubechies2003,mallat2008wavelet} for theories of frames and wavelet frames, \cite{ShenICM2010,Dong2015} for a short survey on the theory and applications of frames, and \cite{Dong2010IASNotes} for a more detailed survey.

A set $X=\{g_j: j\in\Z\}\subset L_2(\mathbb{R}^{\rm dim})$, with ${\rm dim}\in \N$, is
called a frame of $L_2(\mathbb{R}^{\rm dim})$ if
\begin{equation*}
A\|f\|_{L_2(\R^{\rm dim})}^2\le\sum_{j\in \Z}|\langle f,g_j \rangle|^2\le B \|f\|_{L_2(\R^{\rm dim})}^2, \quad \forall f \in
L_2(\mathbb{R}^{\rm dim}),
\end{equation*}
where $\langle \cdot,\cdot \rangle$ is the inner product of
$L_2(\mathbb{R}^{\rm dim})$. We call $X$ a tight frame if it is a frame with $A=B=1$.
For any given frame $X$ of $L_2(\R^{\rm dim})$, there exists another frame
$\widetilde X=\{\widetilde g_j: j\in\Z\}$ of $L_2(\R^{\rm dim})$ such that
$$f=\sum_{j\in \Z}\langle f,g_j \rangle \widetilde g_j \quad \forall f \in
L_2(\mathbb{R}^{\rm dim}).$$ We call $\widetilde X$ a dual frame of $X$. We shall
call the pair $(X,\widetilde X)$ bi-frames. When $X$ is a tight frame, we have
$$f=\sum_{j\in \Z}\langle f,g_j \rangle g_j \quad \forall f \in
L_2(\mathbb{R}^{\rm dim}).$$

For given $\Psi=\{\psi_1,\ldots,\psi_J\}\subset L_2(\mathbb{R}^{\rm dim})$, the
corresponding quasi-affine system $X^N(\Psi)$, $N\in\Z$ generated by $\Psi$ is defined by
the collection of the dilations and the shifts of $\Psi$ as
\begin{equation}\label{E:Quasi:Affine}
X^N(\Psi)=\{\psi_{j,n,\bk}:\ 1\le j\le J; n\in\mathbb{Z},
\bk\in\Z^{\rm dim}\},
\end{equation}
where $\psi_{j,n,\mathbf{k}}$ is defined by
\begin{equation}\label{E:Quasi:Affine:1}
\psi_{j,n,\mathbf{k}}=\left\{\begin{array}{cc}
2^{\frac{n\cdot{\rm dim}}2}\psi_{j}(2^n\cdot-\mathbf{k}),&n\ge N;\\
2^{(n-\frac{N}{2})\cdot{\rm dim}}\psi_{j}(2^n\cdot-2^{n-N}\mathbf{k}),&n<N.
\end{array}\right.
\end{equation}
When $X^N(\Psi)$ forms a (tight) frame of $L_2(\mathbb{R}^{\rm dim})$, each function
$\psi_{j}$, $j=1,\ldots, J$, is called a (tight) framelet and the whole
system $X^N(\Psi)$ is called a (tight) wavelet frame system. Note that in the
literature, the affine system is commonly used, which
corresponds to the decimated wavelet (frame) transforms. The quasi-affine
system, which corresponds to the so-called undecimated wavelet (frame)
transforms, was first introduced and analyzed by \cite{ron1997affine}. Here,
we only discuss the quasi-affine system \eqref{E:Quasi:Affine:1}, since it
works better in image restoration and its connection to variational models and PDEs is more natural
than the affine system \cite{CDOS2011,DJS2013,CDS2014}. For simplicity, we denote $X(\Psi):=X^0(\Psi)$ and will focus on $X(\Psi)$ for the rest of this subsection. We will return to the generic quasi-affine system $X^N(\Psi)$ later when needed.

The constructions of framelets $\Psi$, which are desirably (anti)symmetric
and compactly supported functions, are usually based on a multiresolution
analysis (MRA) that is generated by some refinable function $\phi$ with
refinement mask $\mathbf{p}$ and its dual MRA generated by $\widetilde\phi$ with
refinement mask $\widetilde\bp$ satisfying
\begin{equation*}
\phi=2^{\rm dim}\sum_{\mathbf{k}\in\Z^{\rm dim}}{\mathbf{p}[\mathbf{k}]\phi(2\cdot-\mathbf{k})}\quad\mbox{and}\quad
\widetilde\phi=2^{\rm dim}\sum_{\mathbf{k}\in\Z^{\rm dim}}{\widetilde\bp[\mathbf{k}]\widetilde\phi(2\cdot-\mathbf{k})}.
\end{equation*}
The idea of an MRA-based construction of bi-framelets $\Psi=\{\psi_1,\ldots,
\psi_J\}$ and $\widetilde\Psi=\{\widetilde\psi_1,\ldots, \widetilde\psi_J\}$
is to find masks $\bq^{(j)}$ and $\tilde\bq^{(j)}$, which are finite sequences,
such that, for $j=1, 2, \ldots, J$,
\begin{equation}\label{psimaskphi}
\psi_{j} =
2^{\rm dim}\sum_{\mathbf{k}\in\Z^{\rm dim}}{\mathbf{q}^{(j)}[\mathbf{k}]\widetilde\phi(2\cdot-\mathbf{k})}\quad\mbox{and}\quad
\widetilde\psi_{j} = 2^{\rm dim}\sum_{\mathbf{k}\in\Z^{\rm dim}}{\wt
{\mathbf{q}}^{(j)}[\mathbf{k}]\phi(2\cdot-\mathbf{k})}.
\end{equation}
For a sequence $\{\bp[\bk]\}_{\bk}$ of real numbers, we use
$\widehat{\bp}(\gbo)$ to denote its Fourier series: $$\widehat{\bp}(\gbo)= \sum_{\bk \in \Z^{\rm dim}}\bp[\bk] e^{-i\bk\cdot\gbo}.$$

The mixed extension principle (MEP) of \cite{ron1997affineII} provides a general theory
for the construction of MRA-based wavelet bi-frames. Given two sets of finitely supported masks $\{\bp,\bq^{(1)},\ldots,\bq^{(J)}\}$ and
$\{\tilde\bp,\tilde\bq_1,\ldots,\tilde\bq_J\}$, the MEP says that as long as
we have
\begin{equation}\label{MEPCondition}
\widehat{\bp}(\bm{\xi})\overline{\widehat{\widetilde\bp}(\bm{\xi})}
+\sum_{j=1}^{J}\widehat{\bq}^{(j)}
(\bm{\xi})\overline{\widehat{\widetilde\bq}^{(j)}(\bm{\xi})}=1
\quad\text{and}\quad
\widehat{\bp}(\bm{\xi})\overline{\widehat{\widetilde\bp}(\bm{\xi}+\bm{\nu})}+\sum_{j=1}^{J}
\widehat{\bq}^{(j)}(\bm{\xi})\overline{\widehat{\widetilde\bq}^{(j)}(\bm{\xi}+\bm{\nu})}=0,
\end{equation}
for all $\bm{\nu}\in\{0,\pi\}^{\rm dim}\setminus\{{\bf 0}\}$ and
$\bm{\xi}\in[-\pi,\pi]^{\rm dim}$, the quasi-affine systems $X(\Psi)$ and
$X(\widetilde{\Psi})$ with $\Psi$ and $\widetilde{\Psi}$ given by
\eqref{psimaskphi} forms a pair of \textit{bi-frames} in $\LRs$. In
particular, when $\bp=\widetilde\bp$ and $\bq^{(j)}=\widetilde\bq^{(j)}$ for
$j=1,\ldots,J$, the MEP \eqref{MEPCondition} become the following unitary
extension principle (UEP) discovered in \cite{ron1997affine}:
\begin{equation}\label{UEPCondition}
|\widehat{\bp}(\bm{\xi})|^2+\sum_{j=1}^{J}{|\widehat{\bq}^{(j)}(\bm{\xi})|^2}=1 \quad\text{and}\quad
\widehat{\bp}(\bm{\xi})\overline{\widehat{\bp}(\bm{\xi}+\bm{\nu})}+\sum_{j=1}^{J}
\widehat{\bq}^{(j)}(\bm{\xi})\overline{\widehat{\bq}^{(j)}(\bm{\xi}+\bm{\nu})}=0,
\end{equation}
and the system $X(\Psi)$ is a \textit{tight frame} of $\LRs$. Here, $\bp$ and $\wt \bp$ are
lowpass filters and $\bq^{(j)}, \wt \bq^{(j)}$ are highpass filters. These filters generate discrete bi-frame (or tight frame if UEP is satisfied) system for the sequence space $\ell_2(\Z^{\rm dim})$.

Now, we show two simple but useful examples of univariate tight framelets.

\begin{example}\label{example:haar} Let $\bp =
\frac12[1,1]$ be the refinement mask of the piecewise constant B-spline
$B_1(x)=1$ for $x\in[0,1]$ and $0$ otherwise. Define $\bq^{(1)} = \frac12[1,
-1]$. Then ${\bp}$ and ${\bq^{(1)}}$ satisfy both identities of
\eqref{UEPCondition}. Hence, the system $X(\psi_1)$ defined in
\eqref{E:Quasi:Affine} is a tight frame of $\LR$.
\end{example}

\begin{example}\label{example:linear} \cite{ron1997affine}. Let $\bp =
\frac14[1,2,1]$ be the refinement mask of the piecewise linear B-spline
$B_2(x)=\max{(1-|x|,0)}$. Define $\bq^{(1)} = \frac{\sqrt{2}}{4}[1,0,-1]$ and
$\bq^{(2)} = \frac14[-1,2,-1]$. Then ${\bp}$, ${\bq^{(1)}}$ and ${\bq^{(2)}}$ satisfy
both identities of \eqref{UEPCondition}. Hence, the system $X(\Psi)$ where
$\Psi=\{\psi_1,\psi_2\}$ defined in \eqref{E:Quasi:Affine} is a tight frame
of $\LR$.
\end{example}

In the discrete setting, let an image $\mathbf{f}$ be a ${\rm dim}$-dimensional array. We denote the fast $(L+1)$-level wavelet frame transform/decomposition with filters $\{\bq^{(0)}=\bp, \bq^{(1)}, \cdots, \bq^{(J)}\}$  (see, e.g., \cite{Dong2010IASNotes}) as
\begin{equation}\label{D:Wu}
\mathbf{Wu}=\{\mathbf{W}_{j,l}\mathbf{u}: (j,l)\in\mathbb{B}\},
\end{equation}
where $$\mathbb{B}=\{(j,l):\ 1\le j\le J, 0\le l\le L\}\cup\{(0,L)\}.$$ The wavelet frame coefficients of $\mathbf{u}$ are computed by $\mathbf{W}_{j,l}\mathbf{u}=\bq_{j,l}[-\cdot]\circledast \mathbf{u}$, where $\circledast$
denotes the convolution operator with a certain boundary condition, e.g.,
periodic boundary condition, and $\bq_{j,l}$ is defined as
\begin{equation}\label{dilatedfilter}
\bq_{j,l}=\check{\bq}_{j,l}\circledast \check{\bq}_{0,l-1}
\circledast \ldots \circledast \check{\bq}_{0,0}\quad\mbox{with}\quad
\check{\bq}_{j,l}[\mathbf{k}]
=\left\{\begin{array}{rl}\bq^{(j)}[2^{-l}\mathbf{k}],&\mathbf{k}\in2^{l}\Z^{\rm dim};\\
0,&\mathbf{k}\notin2^{l}\Z^{\rm dim}.\end{array}\right.
\end{equation}
Similarly, we can define $\widetilde{\mathbf{W}}\bu$ and
$\widetilde{\mathbf{W}}_{j,l}\mathbf{u}$ given a set of dual filters $\{\tilde\bq^{(0)}=\tilde\bp,\tilde\bq^{(1)},\ldots,\tilde\bq^{(J)}\}$. We denote the inverse wavelet frame
transform (or wavelet frame reconstruction) as $\widetilde{\mathbf{W}}^\top$,
which is the adjoint operator of $\widetilde{\mathbf{W}}$, and by the MEP, we
have the perfect reconstruction formula
\begin{equation*}
\mathbf{u}=\widetilde{\mathbf{W}}^\top \mathbf{Wu}, \quad\mbox{for all }\mathbf{u}.
\end{equation*}
In particular when $\mathbf{W}$ is the transform for a tight frame system, the
UEP gives us
\begin{equation}\label{E:Perfect:Reconstruction}
\mathbf{u}={\mathbf{W}}^\top \mathbf{Wu}, \quad\mbox{for all
}\mathbf{u}.
\end{equation}
In this paper, we will focus our analysis on the case ${\rm dim}=2$, i.e. for 2-dimensional images/functions. Also, we will only consider single-level wavelet frame transforms. For this case, we simply have $$\mathbf W\mathbf u=\{\bq^{(j)}[-\cdot]\circledast \mathbf{u}: 0\le j\le J\}.$$

\subsection{Notation, Assumptions and Simple Facts}

Throughout the rest of this paper, we denote $\Psi=\{\psi_1,\ldots,\psi_J\}$ as the set of framelets, denote $\phi$ as the corresponding refinable function, and denote $\{\bq^{(0)}, \bq^{(1)}, \cdots, \bq^{(J)}\}$ as the associated finitely supported filters. The following refinement equations are satisfied
\begin{equation}\label{E:Refinement}
\phi=4\sum_{\mathbf{k}\in\Z^{2}}\bq^{(0)}[\mathbf{k}]\phi(2\cdot-\mathbf{k})\quad\mbox{and}\quad
\psi_{j} = 4\sum_{\mathbf{k}\in\Z^{2}}\mathbf{q}^{(j)}[\mathbf{k}]\phi(2\cdot-\mathbf{k}),
\end{equation}
for $1\le j\le J$. In this paper, we focus on the tensor-product B-spline tight wavelet frame systems constructed by \cite{ron1997affine}, and $\phi$ is a tensor-product B-spline function. We shall refer to the elements in $\Psi$ as B-spline framelets.

We start with the following basic definition:
\begin{definition}\label{Def:IndexSets} Let $\Omega=(0,1)^2\subset\R^2$ and $n\in\N$. Define
\begin{equation}\label{D:IndexSets}
\begin{array}{ll}
 &\mathbf O_n:=\left\{\bk\in\Z^2 : 2^{-n}\mathbb{Z}^2\cap\overline{\Omega}\right\},\cr
 &\mathbf M_n:=\left\{\bk\in\Z^2 : \emph{supp}(\phi_{n,\bk})\subset\overline{\Omega}\right\},\cr
 &\mathbf K_n:=\left\{\bk\in\mathbf M_n: \bk+C\cdot\emph{supp}(\bq^{(j)})\in\mathbf M_n \mbox{ for all } 0\le j\le J \right\},
\end{array}
\end{equation}
where $$\phi_{n,\bk}=2^n\phi(2^n\cdot-\bk).$$ Note that the index set $\mathbf K_n$ (in particular, the constant $C$) is defined such that the boundary condition of $\left(\bq^{(j)}[-\cdot]\circledast \mathbf{u}\right)[\bk]$ and its high-order analogue(s) (see \eqref{Boundary}) are inactive for $\bk\in\mathbf K_n$.
\end{definition}

Given $\{\phi_{n,\bk}:\bk\in\mathbb{Z}^2\}$, define the associated sampling operator as
\begin{equation}\label{D:Tn}
(\mathbf T_n f)[\bk]:=2^n \left<f,\phi_{n,\bk}\right>_{L_2(\Omega)}\qquad\mbox{for } f\in L_2(\Omega)\ \mbox{ and }\ \bk\in\mathbf K_n.
\end{equation}
In particular, we assume that image $\mathbf u_n\in\R^{|\mathbf K_n|}$ is sampled from its continuum counterpart $u\in L_2(\Omega)$ by $\mathbf u_n=\mathbf T_n u$. Therefore, when the undecimated wavelet frame transforms are applied to $\mathbf u_n$, the underlying quasi-affine system we use is $X^n(\Psi)$ (see \eqref{E:Quasi:Affine} and \eqref{E:Quasi:Affine:1} for the definition of $X^N(\Psi)$). Note that if $X^n(\Psi)$ is used, we have
$$\psi_{j,n-1,\bk}=2^{n-2}\psi_j(2^{n-1}\cdot-\bk/2).$$

We write the standard single-level wavelet transform as $$W_n f[j,\bk] := 2^n\left<f,\psi_{j,n-1,\bk}\right>.$$ By the refinement equation \eqref{E:Refinement}, we have $\psi_{j,n-1,\bk}=\sum_{\mathbf l\in\mathbb{Z}^2}\bq^{(j)}[\mathbf l-\bk]\phi_{n,\mathbf l}$, and hence
\begin{eqnarray*}
W_n f[j,\bk]&=&2^n\left<f,\psi_{j,n-1,\bk}\right>=2^n\left<f,\sum_{\mathbf l\in\mathbb{Z}^2}\bq^{(j)}[\mathbf l-\bk]\phi_{n,\mathbf l}\right> \cr
&=&2^n\sum_{\mathbf l\in\mathbb{Z}^2}\bq^{(j)}[\mathbf l-\bk]\left<f,\phi_{n,\mathbf l}\right> \cr
&=&\left(\bq^{(j)}[-\cdot]\circledast \mathbf T_nf\right)[\bk],\qquad \mbox{for }\bk\in\mathbf K_n.
\end{eqnarray*}

Recall from \cite{CDOS2011} that given B-spline framelets $\psi_j$, there exists a compactly supported function $\varphi_j$ with $\emph{supp}(\varphi_j)=\emph{supp}(\psi_j)$ and $c_j:=\int\varphi_j\ne0$, such that $D_j\varphi_j=\psi_j$, where $D_j$ is the differential operator associated to $\psi_j$. Then, it is not hard to see that
\begin{equation}\label{E:IntegratingPsi}
D_j\varphi_{j,n-1,\bk}=2^{s_j(n-1)}\psi_{j,n-1,\bk}
\end{equation}
with $s_j$ the order of $D_j$. Therefore, the wavelet frame transform $W_nf$ can be regarded as a sampling of the derivatives:
\begin{eqnarray*}
W_n f[j,\bk]&=&2^n\left<f,\psi_{j,n-1,\bk}\right>=2^{n-s_j(n-1)}\left<f,D_j\varphi_{j,n-1,\bk}\right>\cr
&=&(-1)^{s_j}2^{n-s_j(n-1)}\left<D_jf,\varphi_{j,n-1,\bk}\right>,\qquad\mbox{for }\bk\in\mathbf K_n.
\end{eqnarray*}
Thus, we have, for $\bk\in\mathbf K_n$,
\begin{equation}\label{E:Relation:Wn':D':Unweighted}
\left(\bq^{(j)}[-\cdot]\circledast \mathbf T_nf\right)[\bk]=(-1)^{s_j}2^{n-s_j(n-1)}\left<D_jf,\varphi_{j,n-1,\bk}\right>.
\end{equation}

Observe from the general model \eqref{uni}, the variable $\mathbf v$ has the same structure as the wavelet frame coefficients, which makes it a vector-valued array with $J$ components where $J$ is the total number of wavelet frame bands. We start with the following definition of vector-valued function/sequence spaces.
\begin{definition}\label{norm}
Suppose $f$ is a vector-valued function on a (continuum or discrete) domain $\Omega$, i.e. for each $x\in\Omega$, $f(x)$ is specified as a vector in the Euclidean space $(\mathbb R^{J^r},\|\cdot\|_{\ell_q})$ with $r=0,1,2$. Let $W(\Omega)$ be a certain Banach space (such as an $L_p$, a Sobolev space or an $\ell_p$ space). Define
\begin{equation}\label{norm_def}
\|f\|_{W(\Omega;B)}=\|f_B\|_{W(\Omega)},
\end{equation}
where $f_B$ is the (almost everywhere defined) function such that
\begin{equation*}
f_B(x)=\|f(x)\|_B.
\end{equation*}
Note that we may only mention the norm (\emph{e.g.} "$L_p(\Omega;\ell_2)$") or the space (\emph{e.g.} "$W_s^p(\Omega;\mathbb R^J)$") whenever there is no confusion.
\end{definition}

For vector-valued function $v\in L_2\left(\Omega;\mathbb R^J\right)$, $v=(v_1,\ldots,v_J)$, we can define the sampling operator $\mathbf S_n: L_2\left(\Omega;\R^{J}\right)\to\R^{J|\mathbf K_n|}$ as follows
\begin{equation}\label{D:Sn}
\left(\mathbf S_nv\right)[j;\bk]=2^n\left<v_j,c_j^{-1}\varphi_{j,n-1,\bk}\right>,\quad \bk\in\mathbf K_n, 1\le j\le J,
\end{equation}
where $\varphi_{j,n-1,\bk}=2^{n-2}\varphi_j(2^{n-1}\cdot-\bk/2)$. In particular, we assume that image $\mathbf v_n\in\R^{J|\mathbf K_n|}$ is sampled from its continuum counterpart $v\in L_2\left(\Omega;\R^{J}\right)$ by $\mathbf v_n=\mathbf S_n v$. One can verify that there exists a compactly supported function $\varphi_{ij, n-2,\bk}$, with
\begin{equation}\label{phi:ij}
\varphi_{ij, n-2,\bk}=2^{n-4}\varphi_{ij}(2^{n-2}\cdot-\bk/4)
\end{equation}
such that
\begin{equation*}
\left(\bq^{(i)}[-\cdot]\circledast c_j^{-1}\varphi_{j,n-1,\cdot}\right)[\bk]=2^{-\left(n-2\right)s_i}D_i\varphi_{ij, n-2,\bk},
\end{equation*}
where $s_i$ is the order of $D_i$.

The existence of the above $\varphi_{ij,n-2,\bk}$ is due to the vanishing moment of $\bq^{(i)}$. Let $\mathbf s_i=(s_{1i}, s_{2i})$ be the vanishing moment of $\bq^{(i)}$. Observe that
\begin{equation*}
\left(\widehat{\left(\bq^{(i)}[-\cdot]\circledast c_j^{-1}\varphi_{j,0,\cdot}\right)}[\mathbf 0]\right)(\bm\xi)=\widehat{\bq}^{(i)}(\bm\xi)\cdot\ \widehat{\left(c_j^{-1}\varphi_{j,0,\mathbf 0}\right)}(\bm\xi)
\end{equation*}
At the right-hand-side
\begin{equation*}
\widehat{\bq}^{(i)}(\bm\xi)\propto(i\bm\xi)^{\mathbf{s}_i}
\end{equation*}
near $\bm\xi=\mathbf 0$, and
\begin{equation*}
\widehat{\left(c_j^{-1}\varphi_{j,0,\mathbf 0}\right)}(\mathbf 0)=1
\end{equation*}
Define $\varphi_{ij, -1,\mathbf 0}$ by
\begin{equation*}
\widehat{\varphi}_{ij, -1,\mathbf 0}:=(i\xi)^{-\mathbf{s}_i}\widehat{\bq}^{(i)}(\bm\xi)\cdot\ \widehat{\left(c_j^{-1}\varphi_{j,0,\mathbf 0}\right)}(\bm\xi)
\end{equation*}
then
\begin{equation*}
c_{ij}:=\int\varphi_{ij, -1,\mathbf 0}<\mathcal 1
\end{equation*}
According to our convention, the integral remains the same when the dilation level $-1$ is replaced by other values. By Schwartz-Paley-Wiener theorem, we have
\begin{equation*}
\left|\widehat{\bq}^{(i)}(\bm\xi)\cdot\ \widehat{\left(c_j^{-1}\varphi_{j,0,\mathbf 0}\right)}(\bm\xi)\right|\leq C_1\cdot (1+|\bm\xi|)^{C_2} e^{C^{(i)}\cdot \emph{\emph{supp}}(\varphi_{j,0,\mathbf0})|\emph{\emph{Im}}\bm\xi|},
\end{equation*}
for $\bm\xi\in\mathbb{C}^2$. Therefore,
\begin{equation*}
|\widehat{\varphi}_{ij,0,\bm{0}}(\bm\xi)|\leq C'_1\cdot (1+|\bm\xi|)^{C_2-s_i} e^{C^{(i)}\cdot \emph{\emph{supp}}(\varphi_{j,0,\mathbf0})|\emph{\emph{Im}}\bm\xi|},
\end{equation*}
for $\bm\xi\in\mathbb{C}^2$. Consequently
\begin{equation}\label{Boundary}
\emph{\emph{supp}}(\varphi_{ij, -1,\mathbf 0})\subset C\cdot \emph{\emph{supp}}(\varphi_{j,0,\mathbf0}).
\end{equation}

Finally we shall vary $n$ and $\bk$ and get the whole set of functions according to equation (\ref{phi:ij}). It is worth clarifying that the constant $C$ in definition (\ref{Def:IndexSets}) can be taken as $C:=\max_{i}C^{(i)}$.

%\begin{definition}\label{sample}
%Let $\tau\in L_1\left(\Omega;\mathbb R^{J^d}\right)$
%The "sampling of derivatives" $\mathbf T^{(d)}_n:L_1\left(\Omega;\mathbb R^{J^d}\right)\to\mathbb{R}^{JK_n}$: Given $\tau\in L_1\left(\Omega;\mathbb R^{J^d}\right)$ define
%\begin{equation}
%\left(\mathbf T^{(d)}_n\tau\right)\left[I;k\right]=2^n\left<\tau_I,c^{-1}_I\varphi_{I,n-d,k}\right>
%\end{equation}
%in which $d=|I|$ is the length of the multi-index $I$, and $c_I=\int\varphi_I$.
%
%In particular $\mathbf T^{(0)}_n=\mathbf T_n$, and $\mathbf T^{(1)}_n=\mathbf S_n$, where
%\begin{equation}
%\left(\mathbf S_nv\right)[j;k]=2^n\left<v_j,c_j^{-1}\varphi_{j,n-1,k}\right>
%\end{equation}
%for $v\in L_1\left(\Omega;\mathbb R^J\right)$.
%\end{definition}

%For deriving the channelwise wavelet transform of this sampling of $v$, we first figure out the effect of the convolution operator in the current context. Observe that after iterated integration to order-$s_j$ the result function $\varphi_j$ turns out to be without any more vanishing moments (that is, $c_j=\int\varphi_j$ is non-zero), there exists functions $\varphi_{ij, n-2,k}$ (defined in the quati-affine context) such that:
%$$\left(h_i[-\cdot]*c_j^{-1}\varphi_{j,n-1,\cdot}\right)[k]=2^{-\left(n-2\right)|s_i|}D_i\varphi_{ij, n-2,k}$$

\begin{definition}\label{Def:Wn':Wn''} Define the weighted discrete wavelet transforms $\mathbf W_n':\mathbb{R}^{|\mathbf K_n|}\to\mathbb{R}^{J|\mathbf K_n|}$ and $\mathbf W_n'':\mathbb{R}^{J|\mathbf K_n|}\to\mathbb{R}^{J^2|\mathbf K_n|}$ respectively as:
\begin{eqnarray}\label{W_n}
(\mathbf W'_n\mathbf u_n)[j;\bk]&=&\lambda_{j}'\left(\bq^{(j)}[-\cdot]\circledast\mathbf u_n\right)[\bk],\quad 1\le j\le J;\cr
(\mathbf W_n''\mathbf t_n)[i,j;\bk]&=&\lambda_{ij}''\left(\bq^{(i)}[-\cdot]\circledast\mathbf v_n[j;\cdot]\right)[\bk],\quad 1\le i,j\le J.
\end{eqnarray}
In \eqref{W_n}, the weights $\lambda'_j$ and $\lambda''_{i,j}$ are chosen as $$\lambda_{j}'=c_{j}^{-1}(-1)^{s_j}2^{(n-1)s_j}\quad\mbox{and}\quad\lambda_{ij}''=c_{ij}^{-1}(-1)^{s_{i}}2^{(n-2)s_{i}}.$$ As we will see from below that the values $\{s_j\}_{1\le j\le J}\subset\N$ are the orders of the (partial) differential operators in $\mathbf D'$ and $\mathbf D''$.
\end{definition}

Let $\mathbf D'$ and $\mathbf D''$ be given by Definition \ref{Def:D':D''} with $s_j:=|D_j|$. By \eqref{E:Relation:Wn':D':Unweighted}, we have
\begin{equation}\label{E:Relation:Wn':D'}
(\mathbf W_n'\mathbf T_nu)[j;\bk]=2^{n}\left<D_ju,c_j^{-1}\varphi_{j,n-1,\bk}\right>,\quad\mbox{for } u\in L_2(\Omega).
\end{equation}
Observe that
\begin{eqnarray}\label{W_nS_nv}
\left(\bq^{(i)}[-\cdot]\circledast2^n\left<v_j,c_j^{-1}\varphi_{j,n-1,\cdot}\right>\right)[\bk]
&=&2^n\left<v_j,\left(\bq^{(i)}[-\cdot]\circledast c_j^{-1}\varphi_{j,n-1,\cdot}\right)[\bk]\right>\cr
&=&2^n\left<v_j,2^{-\left(n-1\right)s_i}D_i\varphi_{ij, n-2,\bk}\right>\cr
&=&(-1)^{s_i}2^{n-\left(n-1\right)s_i}\left<D_iv_j,\varphi_{ij, n-2,\bk}\right>
\end{eqnarray}
Therefore,
\begin{equation}\label{E:Relation:Wn'':D''}
(\mathbf W_n''\mathbf S_nv)[i,j;\bk]=2^n\left<D_iv_j,c_{ij}^{-1}\varphi_{ij, n-2,\bk}\right>,\quad v\in L_2(\Omega;\R^J).
\end{equation}
Furthermore,
\begin{eqnarray}\label{E:SnDu=Wn'Tnu}
\left(\mathbf S_n\mathbf D'u\right)[j;\bk]&=&2^n\left<D_ju,c_j^{-1}\varphi_{j,n-1,\bk}\right>\cr
&=&2^n\lambda_j'\left<u,\psi_{j,n-1,\bk}\right>\cr
&=&(\mathbf W_n'\mathbf T_nu)[j;\bk].
\end{eqnarray}

\section{Asymptotic Analysis of Model \eqref{main}}\label{S:Theory}

%\subsection{The property of the functionals $\tilde{V}$ and $\tilde{V}_n$}\label{inf-conv}
Recall the definition of the objective function $F_n$ of \eqref{main}:
\begin{eqnarray}\label{F_n_def}
F_n(\mathbf{u}_n,\mathbf{v}_n)=\nu_1\|\mathbf W'_n\mathbf{u}_n-\mathbf{v}_n\|^p_{\ell_p(\mathbf{K}_n;\ell_2(\mathbb R^J))}+\nu_2\|\mathbf W''_n\mathbf{v}_n\|^q_{\ell_q(\mathbf{K}_n;\ell_2(\mathbb R^{J^2}))}+\frac{1}{2}\|\mathbf A_n\mathbf{u}_n-\mathbf{f}_n\|_2^2,
\end{eqnarray}
with $1\le p,q\le 2$. Here, the wavelet frame transforms $\mathbf W_n'$ and $\mathbf W_n''$ are given in Definition \ref{Def:Wn':Wn''}. Operator $\mathbf A_n$ is a discretization of its continuum counterpart $A: L_2(\Omega)\mapsto L_2(\Omega)$ satisfying the following condition:
\begin{equation}\label{Assumption:An:Tn}
\lim_{n\to\infty}\|\mathbf{T}_n Au-\mathbf{A}_n\mathbf{T}_n u\|_2=0\quad\mbox{for all }
u\in L_2(\Omega),
\end{equation}
Note that operator $A$ that corresponds to image denoising, deblurring and inpainting indeed satisfies the above assumption \cite{CDOS2011}.

To study the asymptotic behaviour of the variation model \eqref{main} thoroughly, we first rewrite the objective function $F_n$ to a new one that is defined on a function space instead of a finite dimensional Euclidean space. We regard $\mathbf f_n$, $\mathbf{u}_n$ and $\mathbf{v}_n$ as samples of their continuum counterparts $f\in L_2(\Omega)$ $u\in W_{2s}^p(\Omega)$ and $v\in W_{s}^q(\Omega;\R^J)$, i.e. $\mathbf{f}_n=\mathbf T_nf$, $\mathbf{u}_n=\mathbf T_nu$ and $\mathbf{v}_n=\mathbf S_nv$. Then, we define
\begin{eqnarray}\label{E_n_def}
E_n(u,v)&:=&\nu_1\|\mathbf W'_n\mathbf T_nu-\mathbf S_nv\|^p_{\ell_p(\mathbf{K}_n;\ell_2(\mathbb R^J))}+\nu_2\|\mathbf W''_n\mathbf S_nv\|^q_{\ell_q(\mathbf{K}_n;\ell_2(\mathbb R^{J^2}))}+\frac{1}{2}\|\mathbf A_n\mathbf T_nu-\mathbf T_nf\|_2^2\cr
&=:&\tilde{V}_n(u,v)+\frac{1}{2}\|\mathbf A_n\mathbf T_nu-\mathbf T_nf\|_2^2
\end{eqnarray}

The following proposition ensures that the original problem $\inf_{\mathbf u_n,\mathbf v_n} F_n(\mathbf u_n, \mathbf v_n)$ is equivalent to the new problem $\inf_{u,v} E_n(u,v)$.

\begin{proposition}\label{E_nF_n} For any given $\mathbf{u}_n$, $\mathbf{v}_n$, there exists $u\in W_{2s}^p(\Omega)$ and $v\in W_{s}^q(\Omega;\mathbb{R}^J)$ such that $E_n(u,v)=F_n(\mathbf{u}_n,\mathbf{v}_n)$. Conversely, for any $u\in W_{2s}^p(\Omega)$ and $v\in W_{s}^q(\Omega;\mathbb{R}^J)$, there exist $\mathbf{u}_n$ and $\mathbf{v}_n$ such that $E_n(u,v)=F_n(\mathbf{u}_n,\mathbf{v}_n)$. In particular, we have $$\inf_{\mathbf u_n\in\R^{|\mathbf K_n|},\mathbf v_n\in\R^{J|\mathbf K_n|}} F_n(\mathbf u_n,\mathbf v_n)=\inf_{u\in W_{2s}^p(\Omega),v\in W_{s}^q(\Omega;\R^J)} E_n(u,v).$$
\end{proposition}

To prove Proposition \ref{E_nF_n}, the following lemma is needed.
\begin{lemma}\label{onto} Given the sets of functions $\{\phi_{n,\bk}:\bk\in\mathbf{K}_n\}$ and $\{\varphi_{j,n-1,\bk}: \bk\in\mathbf{K}_n\}$ for any $1\le j\le J$, there exist dual functions $\{\tilde{\phi}_{n,\bk}:\bk\in\mathbf{K}_n\}$ and $\{\tilde{\varphi}_{j,n-1,\bk}:\bk\in\mathbf{K}_n\}$, with any prescribed smoothness, whose translations satisfy the relations $$\left<\tilde{\phi}_{n,\bk'}, \phi_{n,\bk}\right>=\delta_{\bk'\bk}\quad\mbox{and}\quad\left<\tilde{\varphi}_{j,n-1,\bk'}, \varphi_{j,n-1,\bk}\right>=\delta_{\bk'\bk}.$$
\end{lemma}

\begin{proof} The proof of existence of dual for $\{\phi_{n,\bk}:\bk\in\mathbf{K}_n\}$ was given by \cite[Proposition 3.1]{CDOS2011}. Therefore, we focus on the proof of $\{\varphi_{j,n-1,\bk}: \bk\in\mathbf{K}_n\}$ for any $1\le j\le J$.

Since $\varphi_{j,n-1,0}$ has a compact support, $\hat\varphi_{j,n-1,0}$ is analytic, and thus has only isolated zeros. Consequently, for any $\hat{\mathbf a}\in L_2([-\pi, \pi]^2)$, $\hat{\mathbf a}(\bm{\xi})\hat{\varphi}_{j,n-1,0}(\bm{\xi})=0$ implies $\hat{\mathbf a}(\bm\xi)=0$. In particular, given any sequence $\mathbf a\in \ell_0(\Z^2)$, $\sum_{\bk\in\mathbb{Z}^2}\mathbf a[\bk]\varphi_{j,n-1,\bk}=0$ implies all coefficients $\mathbf a[\bk]=0$, i.e. the given system is finite linear independent. Let $\sigma$ be some compactly supported function with a certain given smoothness, and $f^\sigma=f*\sigma$. Note that $\{\varphi^\sigma_{j,n-1,\bk}:\bk\in\mathbf{K}_n\}$ is a linearly independent set since $\hat{\varphi^\sigma}_{j,n-1,\bk}$ is analytic. Consequently, $\varphi^\sigma_{j,n-1,\bk_0}\not\in span\{\varphi^\sigma_{j,n-1,\bk}:\bk\in\mathbf{K}_n-\{\bk_0\}\}$, and there exists a function $f_0\in (span\{\varphi^\sigma_{j,n-1,\bk}:\bk\in\mathbf{K}_n-\{\bk_0\}\})^\perp \setminus(\varphi^\sigma_{j,n-1,\bk_0})^\perp\not=\emptyset$, where the total space is $L_2(\mathbb{R}^2)$. Define $\tilde{\varphi}_{j,n-1,\bk_0}=\left(\left<f_0,\varphi^\sigma_{j,n-1,\bk_0}\right>^{-1}f_0\right)^\sigma$, and apply the same process to the other $\bk\in\mathbf K_n$. We obtain the desired result.
\end{proof}

\begin{proof}$\mathbf{(of\ Proposition\ \ref{E_nF_n})}$ For one direction, given $\mathbf{u}_n$, $\mathbf{v}_n$, define $u=2^{-n}\sum_{\bk\in\mathbb{Z}^2}\mathbf{u}_n[\bk]\tilde{\phi}_{n,\bk}$, $v_j=2^{-n}\sum_{\bk\in\mathbb{Z}^2}\mathbf{v}_n[j;\bk]\tilde{\varphi}_{j,n-1,\bk}$, then
\begin{eqnarray*}
(\mathbf T_nu)[\bk]&=&\sum_{\bk'\in\mathbb{Z}^2}\mathbf{u}_n[\bk']\left<\tilde{\phi}_{n,\bk'},\phi_{n,\bk}\right>=\mathbf{u}_n[\bk]\cr
(\mathbf S_nv)[j;\bk]&=&\sum_{\bk'\in\mathbb{Z}^2}\mathbf{v}_n[j;\bk']\left<\tilde{\varphi}_{j,n-1,\bk'},\varphi_{j,n-1,\bk}\right>=\mathbf{v}_n[j;\bk]
\end{eqnarray*}
consequently, $E_n(u,v)=F_n(\mathbf{u}_n,\mathbf{v}_n)$. Conversely, given $u$ and $v$, define $\mathbf{u}_n=\mathbf T_nu$ and $\mathbf{v}_n=\mathbf S_nv$ then obviously, $F_n(\mathbf{u}_n,\mathbf{v}_n)=E_n(u,v)$.
\end{proof}

Consider the variational problem $$\inf_{u\in W_{2s}^p(\Omega),v\in W_{s}^q(\Omega;\R^J)}E(u,v),$$ where
\begin{eqnarray}\label{D:Energy:Variational}
E(u,v)=\nu_1\left\|D'u-v\right\|^p_{L_p(\Omega;\ell_2(\mathbb R^J))}+\nu_2\left\|D''v\right\|^q_{L_q(\Omega;\ell_2(\mathbb R^{J^2}))}+\frac{1}{2}\left\|Au-f\right\|_{L_2(\Omega)}^2.
\end{eqnarray}
Here, the differential operators $\mathbf D'$ and $\mathbf D''$ are defined in Definition \ref{Def:D':D''} with $|\mathbf D'|=|\mathbf D''|=s$. Our main objective of this section is to show the relation between $E_n$ and $E$, and how the solutions of $\inf E_n$ approximates that of $\inf E$. For this, we will use $\Gamma$-convergence \cite{dal1993introduction} as the main tool.

\begin{definition}\label{D:GammaConvergence}
The sequence of functionals $\{E_n\}$ defined on a Banach space $B$ is said to be $\Gamma$-convergent to the functional $E$ if:
\begin{enumerate}
\item $\ \lim\inf_{n\to\mathcal{1}}E_n(u_n)\geq E(u)$  for arbitrary $u_n\stackrel{B}{\longrightarrow}u$;

\item For arbitrary $u\in B$, there exists $u'_n\stackrel{B}{\longrightarrow}u$ such that: $\lim\sup_{n\to\mathcal{1}}E_n(u'_n)\leq E(u)$.
\end{enumerate}
\end{definition}

To show that $E_n$ given by \eqref{E_n_def} indeed $\Gamma$-converges to the functional $E$ given by \eqref{D:Energy:Variational}, we will use the following two lemmas, which show that $E_n$ converges to $E$ pointwise and $\{E_n\}$ is equicontinuous. The proof of the two lemmas will be postponed to the later part of this section.

\begin{lemma}\label{Lemma:Pointwise}
$E_n$ converges to $E$ pointwise, that is, for $(u,v)\in W_{2s}^p(\Omega)\times W_s^q(\Omega;\mathbb{R}^J)$,
\begin{equation*}
\lim_{n\to\mathcal{1}}E_n(u,v)=E(u,v)
\end{equation*}
\end{lemma}

\begin{lemma}\label{Lemma:Equicontinuous}
$\{E_n\}$ forms an equicontinuous family in the sense that, for any given function $(u,v)\in W_{2s}^p(\Omega)\times W_s^q(\Omega;\mathbb R^J)$, and $\varepsilon>0$, there exists an $\eta>0$ independent from $n$ such that $\left|E_n(u',v')-E_n(u,v)\right|<\varepsilon$ holds for any $(u',v')$ satisfying $\|u'-u\|_{W_{2s}^p(\Omega)}+\|v'-v\|_{W_s^q(\Omega;\mathbb R^J)}<\eta$.
\end{lemma}

\begin{theorem}\label{main_1}
Let $E_n$ be given by \eqref{E_n_def} and $E$ by \eqref{D:Energy:Variational}. Then, for every sequence $(u_n, v_n) \to (u,v)$ in $W_{2s}^p(\Omega)\times W_s^q(\Omega;\mathbb R^J)$, we have $\lim_{n\to+\infty}E_n(u_n,v_n)=E(u,v)$. Consequently, $E_n$ $\Gamma$-converges to $E$ in $W_{2s}^p(\Omega)\times W_s^q(\Omega;\mathbb R^J)$.
\end{theorem}
\begin{proof}
The proof of Theorem \ref{main_1} is essentially the same as \cite[Theorem 3.2]{CDOS2011} once we have Lemma \ref{Lemma:Pointwise} and Lemma \ref{Lemma:Equicontinuous}. However, for completeness, we include the proof here.

By Lemma \ref{Lemma:Pointwise} and Lemma \ref{Lemma:Equicontinuous}, we have that, for an arbitrary given $(u, v)\in W_{2s}^p(\Omega)\times W_s^q(\Omega;\mathbb R^J)$, and $\epsilon>0$,
\begin{enumerate}
\item[(a)] $\lim_{n\to+\infty}|E_n(u,v)-E(u,v)|=0$;
\item[(b)]  there exist an integer $\mathcal{N}$ and $\eta>0$
    satisfying $|E_n(u',v')-E_n(u,v)|<\epsilon$ whenever
    $\|u'-u\|_{W_{2s}^p(\Omega)}+\|v'-v\|_{W_s^q(\Omega;\mathbb R^J)}<\eta$ and $n>\mathcal{N}$.
\end{enumerate}
Note that for arbitrary $(u_n,v_n)\in W_{2s}^p(\Omega)\times W_s^q(\Omega;\mathbb R^J)$, we have
$$
|E_n(u_n,v_n)-E(u,v)|\leq |E_n(u,v)-E(u,v)|+|E_n(u,v)-E_n(u_n,v_n)|.
$$
Let the sequence $(u_n,v_n)\to (u,v)$ in $W_{2s}^p(\Omega)\times W_s^q(\Omega;\mathbb R^J)$, and let $\epsilon>0$ be a
given arbitrary number. On one hand, by (a), there exists an $\mathcal{N}_1$ such that $|E_n(u,v)-E(u,v)|<\epsilon/2$ whenever $n>\mathcal{N}_1$. On the other hand, by (b), there exist $\mathcal{N}$ and $\eta$ such that $|E_n(u',v')-E_n(u,v)|<\epsilon/2$ whenever
$\|u'-u\|_{W_{2s}^p(\Omega)}+\|v'-v\|_{W_s^q(\Omega;\mathbb R^J)}<\eta$ and $n>\mathcal{N}$. Since $(u_n,v_n)\to (u,v)$, there exists $\mathcal{N}_2$ such that $\|u-u_n\|_{W_{2s}^p(\Omega)}+\|v-v_n\|_{W_s^q(\Omega;\mathbb R^J)}<\eta$ whenever $n>\mathcal{N}_2$. Letting $u'=u_n$ and $v'=v_n$ leads to $|E_n(u_n,v_n)-E_n(u,v)|<\epsilon/2$ whenever $n>\max\{\mathcal{N},\mathcal{N}_2\}$. Therefore, we have
$$
|E_n(u_n,v_n)-E(u,v)|\leq \epsilon
$$
whenever $n>\max\{\mathcal{N},\mathcal{N}_1,\mathcal{N}_2\}$. This shows that $\lim_{n\to+\infty}E_n(u_n,v_n)=E(u,v)$, and hence both
conditions given in Definition \ref{D:GammaConvergence} are satisfied. Therefore, $E_n$ $\Gamma$-converges to $E$.
\end{proof}

Recall that $(u^\star,v^\star)$ is an $\epsilon$-optimal solution of the problem $\inf_{u,v} E(u,v)$ if
\begin{equation*}%\label{D:epsilon:optimal1}
E(u^\star,v^\star)\le\inf_{u,v} E(u,v)+\epsilon, \quad\mbox{for some }\epsilon>0.
\end{equation*}
In particular, $0$-optimal solutions of $E_n$ or $E$ will be called minimizers. By Theorem \ref{main_1}, we have the following result describing the relation between the $\epsilon$-optimal solutions of $E_n$ and that of $E$.

\begin{corollary}\label{cluster_point}
Let $(u_n^\star,v_n^\star)$ be an $\epsilon$-optimal solution of $E_n$ for a given $\epsilon>0$ and for all $n$.
If the set $\{(u_n^\star,v_n^\star) : n\}$ has a cluster point $(u^\star, v^\star)$, then $(u^\star, v^\star)$ is an $\epsilon$-optimal solution to $E$. In particular, when $(u_n^\star, v_n^\star)$ is a minimizer of $E_n$ and $(u^\star, v^\star)$ a cluster point of the set $\{(u_n^\star,v_n^\star) : n\}$, then $(u^\star,v^\star)$ is a minimizer of $E$.
\end{corollary}

The rest of this section is dedicated to the proof of Lemma \ref{Lemma:Pointwise} and Lemma \ref{Lemma:Equicontinuous}.

%\begin{corollary}\label{cluster_point}
%Denote
%\begin{equation}
%\Sigma_n=\arg\min E_n(\cdot,\cdot)
%\end{equation}
%If there exists
%\begin{equation}
%(u^*,v^*)\in\bigcap_{n=0}^{\mathcal{1}}\overline{\Sigma_n}
%\end{equation}
%then this point (in the function space) is also a minimum locus of $E(\cdot,\cdot)$, that is, $(u^*,v^*)\in\Sigma=\arg\min E(\cdot,\cdot)$.
%
%The corresponding result holds when each minimization is replaced by its $\varepsilon$-approximate\emph{\cite{CDOS2011}} counterpart.
%\end{corollary}
%\begin{proof}
%Take a subsequence if necessary and we shall assume without loss of generality that
%\begin{equation}
%(u_n^*,v_n^*)\longrightarrow (u^*,v^*)
%\end{equation}
%with respect to the Sobolev norm, where each
%\begin{equation}
%(u_n^*,v_n^*)\in\Sigma_n
%\end{equation}
%is a minimum locus of the corresponding optimization problem. Then
%\begin{eqnarray}
%E(u^*,v^*)&\leq&{\lim\inf}_{n\to\mathcal{1}}E_n(u_n^*,v_n^*)\cr
%&=&{\lim\inf}_{n\to\mathcal{1}}\left(\inf_{u,v} E_n(\cdot,\cdot)\right)\cr
%&\leq&{\lim\sup}_{n\to\mathcal{1}}E_n(u_n',v_n')\cr
%&\leq&E(u',v')=\inf_{u,v}E(\cdot,\cdot)
%\end{eqnarray}
%where the first inequality is ensured by the first property of $\Gamma$-convergence, and the last two inequalities, by the second one. More specifically, we shall fix any $(u',v')\in\Sigma$ at advance, and decide the sequence $(u'_n,v'_n)$ by that existence property.
%\end{proof}

\subsection{Proof of Lemma \ref{Lemma:Pointwise}}\label{SubS:Lemma:PW}

The pointwise convergence of the third term of $E_n$ to that of $E$ has been shown by \cite{CDOS2011} under assumption \eqref{Assumption:An:Tn}. Therefore, we focus on the convergence of the first two terms.
Let us first show that $$\|\mathbf W''_n\mathbf S_nv\|_{\ell_q(\mathbf{K}_n;\ell_2(\mathbb R^{J^2}))}\to\|\mathbf D'' v\|_{L_q(\Omega;\ell_2(\mathbb R^{J^2}))},$$ which will imply $$\|\mathbf W''_n\mathbf S_nv\|^q_{\ell_q(\mathbf{K}_n;\ell_2(\mathbb R^{J^2}))}\to\|\mathbf D'' v\|^q_{L_q(\Omega;\ell_2(\mathbb R^{J^2}))}.$$ Recall from Definition \ref{Def:D':D''} that, given $\mathbf D'$, we have $\mathbf D''=(\mathbf D', \ldots, \mathbf D')$. Then,
\begin{eqnarray*}
&&\left|\|\mathbf D''v\|_{L_q(\Omega;\ell_2(\mathbb R^{J^2}))}-\|\mathbf W''_n\mathbf S_nv\|_{\ell_q(\mathbf{K}_n;\ell_2(\mathbb R^{J^2}))}\right|\cr
&=&\left|\left(\int_\Omega\left(\sum_{i,j=1}^J\left|D'_iv_j\right|^2\right)^{q/2}{\rm d}\mathbf x\right)^{1/q}-\left(2^{-2n}\sum_{\bk\in\mathbf{K}_n}\left(\sum_{i,j=1}^J\left|(\mathbf W''_n\mathbf S_nv)[i,j;\bk]\right|^2\right)^{q/2}\right)^{1/q}\right|\cr
(\mbox{By \eqref{E:Relation:Wn'':D''}})&=&\Bigg|\left(\sum_{\bk\in\mathbf{O}_n}\int_{I_\bk}\left(\sum_{i,j=1}^J\left|D'_iv_j\right|^2\right)^{q/2}{\rm d}\mathbf x\right)^{1/q}\cr
& &\hspace*{1in}-\left(2^{-2n}\sum_{\bk\in\mathbf{K}_n}\left(\sum_{i,j=1}^J\left|\left<D'_iv_j,c_{ij}^{-1}\varphi_{ij,n-2,\bk}\right>\right|^2\right)^{q/2}\right)^{1/q}\Bigg|\cr
&=&\Bigg|\left(\sum_{\bk\in\mathbf{O}_n}\int_{I_\bk}\left(\sum_{i,j=1}^J\left|D'_iv_j\right|^2\right)^{q/2}{\rm d}\mathbf x\right)^{1/q}\cr
& &\hspace*{1in}-\left(\sum_{\bk\in\mathbf{K}_n}\int_{I_\bk}\left(\sum_{i,j=1}^J\left|\sum_{\bk'\in\mathbf{K}_n}\left<D'_iv_j,c_{ij}^{-1}\varphi_{ij,n-2,\bk'}\right>\chi_{I_\bk}\right|^2\right)^{q/2}{\rm d}\mathbf x\right)^{1/q}\Bigg|\cr
&\leq&\left(\sum_{\bk\in\mathbf{K}_n}\int_{I_\bk}\left(\sum_{i,j=1}^J\left|D'_iv_j-\sum_{\bk'\in\mathbf{K}_n}\left<D'_iv_j,c_{ij}^{-1}\varphi_{ij,n-2,\bk'}\right>\chi_{I_\bk}\right|^2\right)^{q/2}{\rm d}\mathbf x\right)^{1/q}\cr&&+\left(\int_{\cup_{\bk\in\mathbf{O}_n\setminus\mathbf{K}_n}I_\bk}\left(\sum_{i,j=1}^J\left|D'_iv_j\right|^2\right)^{q/2}{\rm d}\mathbf x\right)^{1/q}\cr
&\leq&\sum_{i,j=1}^J\left(\left\|D'_iv_j-\sum_{\bk\in\mathbf{K}_n}\left<D'_i v_j,c^{-1}_{ij}\varphi_{ij,n-2,\bk}\right>\chi_{I_\bk}\right\|_{L_q(\Omega)}+\|D'_iv_j\|_{L_q(\cup_{\bk\in\mathbf{O}_n\setminus\mathbf{K}_n}I_\bk)}\right).
\end{eqnarray*}
Here, $I_{\bk}$ is the rectangular domain $[\frac{k_1}{2^n},\frac{k_1+1}{2^n}]\times[\frac{k_2}{2^n},\frac{k_2+1}{2^n}]$ where $\bk=(k_1,k_2)$.
By the approximation lemma \cite[Lemma 4.1]{CDOS2011}, we have, for each $i,j$,
\begin{equation*}
\lim_{n\to\mathcal{1}}\left\|D'_iv_j-\sum_{\bk\in\mathbf{O}_n}\left<D'_iv_j,c_{ij}^{-1}\varphi_{ij,n-2,\bk}\right>\chi_{I_\bk}\right\|_{L_q(\Omega)}=0.
\end{equation*}
Also, the Lebesgue measure $\mathfrak{L}$ of the set $\cup_{\bk\in\mathbf{O}_n\setminus\mathbf{K}_n}I_\bk$ satisfies
\begin{equation*}\mathfrak{L}(\cup_{\bk\in\mathbf{O}_n\setminus\mathbf{K}_n}I_\bk)\leq 4\cdot 2^n\left(\frac{\text{diam}(\text{supp}(\phi_{n,0}))}{2^{-n}}+1\right)\cdot(2^{-n})^2=4c\cdot 2^{-n}.
\end{equation*}
Thus, we have
\begin{equation*}
\lim_{n\to\mathcal{1}}\|D'_iv_j\|_{L_q(\cup_{\bk\in\mathbf{O}_n-\mathbf{K}_n}I_\bk)}=0,
\end{equation*}
since $v\in W_s^q(\Omega;\R^J)$ which implies that $D'_iv_j \in L_q(\Omega)$ (for each $(i,j)\in\{1,\cdots,J\}^2$). Altogether, we have $$\|\mathbf W''_n\mathbf S_nv\|^q_{\ell_q(\mathbf{K}_n;\ell_2(\mathbb R^{J^2}))}\to\|\mathbf D'' v\|^q_{L_q(\Omega;\ell_2(\mathbb R^{J^2}))}.$$

Recall \eqref{E:SnDu=Wn'Tnu} that $\mathbf W'_n\mathbf T_nu=\mathbf S_nDu$. Then, following a similar proof as above by replacing the previous summing index $i,j$ with merely $j$, we have:
\begin{eqnarray*}
&&\left|\left\|\mathbf D'u-v\right\|_{L_p(\Omega;\ell_2(\R^J))}-\|\mathbf W'_n\mathbf T_nu-\mathbf S_nv\|_{\ell_p(\mathbf{K}_n;\ell_2(\R^J))}\right|\cr
&=&\left|\left\|\mathbf D'u-v\right\|_{L_p(\Omega;\ell_2(\R^J))}-\|\mathbf S_n\left(\mathbf D'u-v\right)\|_{\ell_p(\mathbf{K}_n;\ell_2(\R^J))}\right|\cr
&\leq&\sum_{j=1}^J\left\|(Du-v)_j-\sum_{\bk\in\mathbf{O}_n}\left<(Du-v)_j,c_j^{-1}\varphi_{j,n-1,\bk}\right>\chi_{I_\bk}\right\|_{L_p(\Omega)}\cr&&+\sum_{j=1}^J\|D'_ju-v_j\|_{L_p(\cup_{\bk\in\mathbf{O}_n\setminus\mathbf{K}_n}I_\bk)}\cr
&\to&0.
\end{eqnarray*}
Therefore, we have $$\|\mathbf W'_n\mathbf T_nu-\mathbf S_nv\|_{\ell_p(\mathbf{K}_n;\ell_2(\R^J))}^p\to\left\|\mathbf D'u-v\right\|_{L_p(\Omega;\ell_2(\R^J))}^p,$$ which concludes the proof of the lemma.

\subsection{Proof of Lemma \ref{Lemma:Equicontinuous}}\label{SubS:Lemma:EqCont}

We shall focus on the equicontinuity of $$\tilde{V}_n=\nu_1\|\mathbf W'_n\mathbf T_nu-\mathbf S_nv\|^p_{\ell_p(\mathbf{K}_n;\ell_2(\mathbb R^J))}+\nu_2\|\mathbf W''_n\mathbf S_nv\|^q_{\ell_q(\mathbf{K}_n;\ell_2(\mathbb R^{J^2}))},$$ since the equicontinuity of $\frac{1}{2}\|\mathbf A_n\mathbf T_nu-\mathbf T_nf\|_2^2$ has been established in \cite[Proposition 3.2]{CDOS2011}.

Let us begin with the bound of the linear operator $S_n:L_p(\Omega;\mathbb R^J)\to\mathbb{R}^{JK_n}$. We denote $\Lambda_{\bk}:=\emph{supp}(\varphi_{j,n-1,\bk})$. Note that when B-spline framelets are used, $\varphi_{j,n-1,\bk}$ has the same support for different $j$. Now, consider
\begin{eqnarray}
\left\|\mathbf S_nv\right\|_{\ell_p(\mathbf{K}_n;\ell_2(\mathbb R^J))}&=&\left(2^{-2n}\sum_{\bk\in\mathbf{K}_n} \left(\sum_{j=1}^{J} |2^n\left<v_j,c_j^{-1}\varphi_{j,n-1,\bk}\right>|^2\right)^{p/2}\right)^{1/p}\cr
&\leq&(2^{-n/p})^2\left(\sum_{\bk\in\mathbf{K}_n}\left(\sum_{j=1}^J \left|2^n\left<v_j,c_j^{-1}\varphi_{j,n-1,\bk}\right>\right|\right)^{p}\right)^{1/p}\cr
&\leq&(2^{-n/p})^2\cdot2^n\left(\sum_{\bk\in\mathbf{K}_n}\left(\sum_{j=1}^J\left\|v_j\right\|_{L_1(\Lambda_{\bk})}\left\|c_j^{-1}\varphi_{j,n-1,\bk}\right\|_{L_\infty(\Omega)}\right)^p\right)^{1/p}\cr
&=&(2^{-n/p})^2\cdot2^{2(n-1)}\left(\max_{j} \left\|c_j^{-1}\varphi_{j}\right\|_{L_\infty(\Omega)}\right)\left(\sum_{\bk\in\mathbf{K}_n}\left(\sum_{j=1}^J\left\|v_j\right\|_{L_1(\Lambda_\bk)}\right)^p\right)^{1/p}\cr
&\leq&(2^{-n/p})^2\cdot2^{2(n-1)}\left(\max_{j} \left\|c_j^{-1}\varphi_{j}\right\|_{L_\infty(\Omega)}\right)\sum_{j=1}^J\left(\sum_{\bk\in\mathbf{K}_n}\left\|v_j\right\|_{L_1(\Lambda_\bk)}^p\right)^{1/p}\cr
&\leq& C(n) \|v\|_{L_1(\Omega;\ell_p(\R^J))}\leq C'(n)\|v\|_{L_p(\Omega;\ell_2(\R^j))},
\end{eqnarray}where we applied after H\"older's inequality (in the third line) and the fact that
\begin{equation*}
\left\|\varphi_{j,n-1,\bk}\right\|_{L_\infty(\Omega)}=2^{n-2}\left\|\varphi_j\right\|_{L_\infty(\Omega)}
\end{equation*}
which can be easily verified using $\varphi_{j,n-1,\bk}=2^{n-2}\varphi_j\left(2^{n-1}\cdot-\bk/2\right)$.

Consider the family of linear operators $2^{-2n/q}\mathbf W''_n\mathbf S_n:W_s^q(\Omega;\mathbb R^J)\to \ell_{q,2}(\mathbb{Z}^2\times J^2)$ ordered by $n$, where the norm of the latter is generally defined as
\begin{equation*}
\|\mathbf a[\bk;i,j]\|_{\ell_{p,q}}=\left(\sum_{\bk\in\Z^2}\left(\sum_{i,j=1}^J |\mathbf a[\bk;i,j]|^q\right)^{p/q}\right)^{1/p}
\end{equation*}
Based upon the above observation on $\mathbf S_n$, and the boundedness of the operator $\mathbf W''_n$ as a matrix, we have
\begin{equation*}
\|2^{-2n/q}\mathbf W''_n\mathbf S_nv\|_{\ell_{q,2}}\leq C(n)\|v\|_{L_1(\Omega;\ell_q(\mathbb R^J))}\leq C'(n)\|v\|_{W_s^q(\Omega;\ell_2(\R^J))}
\end{equation*}
Since we have proved that, for any given $v\in W_s^p(\Omega;\mathbb R^J)$,
\begin{equation*}
\lim_{n\to\mathcal{1}}\|\mathbf W''_n\mathbf S_nv\|_{\ell_q(\mathbf{K}_n;\ell_2(\mathbb R^{J^2}))}=\|\mathbf D''v\|_{L_q(\Omega;\ell_2(\mathbb R^{J^2}))}
\end{equation*}
we have
\begin{equation*}
\sup_n\|2^{-2n/p}\mathbf W''_n\mathbf S_nv\|_{\ell_{q,2}}=\sup_n\|\mathbf W''_n\mathbf S_nv\|_{\ell_q(\mathbf{K}_n;\ell_2(\mathbb R^{J^2}))}<\mathcal{1}.
\end{equation*}
By resonance theorem, $\{2^{-2n/q}\mathbf W''_n\mathbf S_n\}_{n=1}^\mathcal{1}$ is uniformly bounded by some constant, i.e.
\begin{equation}\label{equi_part_1}
\|\mathbf W''_n\mathbf S_n v\|_{\ell_q(\mathbf K_n;\ell_2(\mathbb R^{J^2}))}\leq C_1\|v\|_{W_s^q(\Omega;\ell_2(\mathbb R^J))}
\end{equation}
based on which the rest is justified by Sobolev's inequality.

Applying the bound of $\mathbf S_n$ again, we have
\begin{eqnarray*}\label{equi_part_2}
\|\mathbf W'_n\mathbf T_nu-\mathbf S_nv\|_{\ell_p(\mathbf{K}_n;\ell_2(\mathbb R^J))}&=&\|\mathbf S_n(\mathbf D' u-v)\|_{\ell_p(\mathbf{K}_n;\ell_2(\mathbb R^J))}\cr
&\leq&\|\mathbf S_n\|_{op}\|(\mathbf D' u-v)\|_{L_p(\Omega;\ell_2(\mathbb R^J))}\cr
&\leq&C_2(n)\left(\|u\|_{W_{2s}^p\left(\Omega\right)}+\|v\|_{W_{s}^q(\Omega;\ell_2(\mathbb R^J))}\right)
\end{eqnarray*}
where the Sobolev's inequality is applied in the last inequality. Since $$\lim_{n\to\infty}\|\mathbf W'_n\mathbf T_nu-\mathbf S_nv\|_{\ell_p(\mathbf{K}_n;\ell_2(\mathbb R^J))}=\left\|\mathbf D'u-v\right\|_{L_p(\Omega;\ell_2(\R^J))},$$ following a similar argument using the resonance theorem, we have
\begin{equation}\label{equi_part_2}
\|\mathbf W'_n\mathbf T_nu-\mathbf S_nv\|_{\ell_p(\mathbf{K}_n;\ell_2(\mathbb R^J))}\leq C\left(\|u\|_{W_{2s}^p\left(\Omega\right)}+\|v\|_{W_{s}^q(\Omega;\ell_2(\mathbb R^J))}\right).
\end{equation}

Observe that $|x^q-y^q|\leq q(\max\{x,y\})^{q-1}|x-y|\leq q(y+|x-y|)^{q-1}|x-y|$ for $x,y\ge0$ and $q\geq 1$. Using (\ref{equi_part_1}) and (\ref{equi_part_2}), for any $\left(u^*,v^*\right)$ in the fixed unit neighborhood $B((u,v);\delta)$ (with $\delta\leq 1$), we have
\begin{eqnarray}
&&\left|\|\mathbf W''_n\mathbf S_nv^*\|^q_{\ell_q(\mathbf K_n;\ell_2(\R^{J^2}))}-\|\mathbf W''_n\mathbf S_nv\|^q_{\ell_q(\mathbf K_n;\ell_2(\R^{J^2}))}\right|\cr
&\leq&q\left(\max\left\{\|\mathbf W''_n\mathbf S_nv^*\|_{\ell_q(\mathbf K_n;\ell_2(\R^{J^2}))},\|\mathbf W''_n\mathbf S_nv\|_{\ell_q(\mathbf K_n;\ell_2(\R^{J^2}))}\right\}\right)^{q-1}\cr
& &\hspace*{2in}\left|\|\mathbf W''_n\mathbf S_nv^*\|_{\ell_q(\mathbf K_n;\ell_2(\R^{J^2}))}-\|\mathbf W''_n\mathbf S_nv\|_{\ell_q(\mathbf K_n;\ell_2(\R^{J^2}))}\right|\cr
&\leq&q\left(C_1\|v\|_{W_{s}^q(\Omega;\ell_2(\mathbb R^J))}+C_1\|v^*-v\|_{W_{s}^q(\Omega;\ell_2(\mathbb R^J))}\right)^{q-1}\cdot C_1\|v^*-v\|_{W_{s}^q(\Omega;\ell_2(\mathbb R^J))}\cr
&\leq&C''(v)\left\|v^*-v\right\|_{W_{s}^q(\Omega;\ell_2(\mathbb R^J))}
\end{eqnarray}
where $\ C''(v)=q\left(\|v\|_{W_{s}^q(\Omega;\ell_2(\mathbb R^J))}+1\right)^{q-1}C_1^q$. By the same argument, we obtain
\begin{eqnarray}
&&\left|\|\mathbf W'_n\mathbf T_nu^*-\mathbf S_nv^*\|^p_{\ell_p(\mathbf K_n);\ell_2(\R^{J})}-\|\mathbf W'_n\mathbf T_nu-\mathbf S_nv\|^p_{\ell_p(\mathbf K_n);\ell_2(\R^{J})}\right|\cr
&\leq& C'(u,v)\left(\|u^*-u\|_{W_{2s}^p(\Omega)}+\|v^*-v\|_{W_{s}^q(\Omega;\ell_2(\R^J))}\right).
\end{eqnarray}
for $\ C'(u,v)=p\left(\|\mathbf D'u-v\|_{L_p(\Omega;\ell_2(\R^J))}+1\right)^{p-1}C_2^p\leq p\left(\|u\|_{W_{2s}^p\left(\Omega\right)}+\|v\|_{W_{s}^q(\Omega;\ell_2(\mathbb R^J))}+1\right)^{p-1}C_2^p$. Therefore,
\begin{eqnarray}
&&\left|\tilde{V}_n(u^*,v^*)-\tilde{V}_n(u,v)\right|\cr
&\leq&C(u,v)\left(\|u^*-u\|_{W_{2s}^p\left(\Omega\right)}+\|v^*-v\|_{W_{s}^q(\Omega;\ell_2(\mathbb R^J))}\right),
\end{eqnarray}
where $C(u,v)=\nu_1C'(u,v)+\nu_2C''(v)$. Since $C=C(u,v)$ does not depend on $n$, we can conclude that $\tilde{V}_n$ is equicontinuous.

\section{Algorithm and Simulations}\label{S:Algorithm}

In this section, we propose an algorithm solving the general model. The algorithm is derived using the idea of the alternating direction method of multipliers (ADMM) \cite{gabay1976dual,bertsekas1989parallel,eckstein1992douglas} which was later rediscovered as the split Bregman algorithm \cite{GoldO,cai2009split}. We also present numerical simulations of the proposed algorithm on one synthetic image and compare it with the analysis based model \eqref{analysis_discrete}. Note that the focus of this paper is to propose the general model and provide a unified asymptotic analysis of the model to draw connections of it with variational models. Therefore, we shall skip convergence analysis of the proposed algorithm and will not provide a comprehensive numerical studies of the algorithm. We only present one example as a proof of concept.

We restrict our attention to the case $p=1, q=1,2$ of the problem \eqref{F_n_def}, which is restated as follows with simplified notation:
$$F(\mathbf{u},\mathbf{v})=\nu_1\|\mathbf W'\mathbf{u}-\mathbf{v}\|_1+\nu_2\|\mathbf W''\mathbf{v}\|^q_q+\frac{1}{2}\|\mathbf A\mathbf{u}-\mathbf{f}\|_2^2.$$
To yield a computationally simple algorithm, we consider the following equivalent problem $\tilde F$ instead
\begin{equation}\label{UnifiedModel:Numerical}
\tilde F(\mathbf{u},\mathbf{v},\mathbf{d},\mathbf{e})=\nu_1\|\bd-\mathbf{v}\|_1+\nu_2\|\be\|^q_q+\frac{1}{2}\|\mathbf A\mathbf{u}-\mathbf{f}\|_2^2+\frac{\mu}{2}(\|\bW' \bu-\bd\|_2^2+\|\bW'' \bv-\be\|_2^2).
\end{equation}
subject to the constraint
\begin{equation*}
\left\{\begin{array}{ll}
\bW'\bu-\bd=0 \\ \bW''\bv-\be=0\\
\end{array}
\right.
\end{equation*}
The augmented Lagrangian of the above problem is
\begin{eqnarray}
L_\mu(\bu,\bv;\bd,\be;\bp,\bq)&=&\nu_1\|\bd-\mathbf{v}\|_1+\nu_2\|\be\|^q_q+\frac{1}{2}\|\mathbf A\mathbf{u}-\mathbf{f}\|_2^2+\frac{\mu}{2}(\|\bW'\bu-\bd\|_2^2+\|\bW''\bv-\be\|_2^2)\cr
& &\hspace*{0.2in}+\mu\left(\left<\bp,\bW'\bu-\bd\right>+\left<\bq,\bW''\bv-\be\right>\right).
\end{eqnarray}
Given a step-length $0\leq\delta<1$, the augmented Lagrangian method \cite{hestenes1969multiplier,powell1969method,glowinski1989augmented} is given as follows
\begin{equation}\label{E:ALM}
\left\{\begin{array}{ll}
(\bu_{k+1},\bv_{k+1},\bd_{k+1},\be_{k+1})=\arg\min_{\bu,\bv,\bd,\be}\{\nu_1\|\bd-\mathbf{v}\|_1+\nu_2\|\be\|^q_q+\frac{1}{2}\|\mathbf A\mathbf{u}-\mathbf{f}\|_2^2\\
\ \ \ \ \ \ \ \ \ \ \ \ \ \ \ \ \ \ \ \ \ \ \ \ \ \ \ \ \ \ \ \ \ \ \ \ \ \ \ \
+\frac{\mu}{2}(\|\bW' \bu-\bd+\bp_k\|_2^2+\|\bW'' \bv-\be+\bq_k\|_2^2)\}\\
\bp_{k+1}=\bp_k+\delta\left(\bW'\bu_{k+1}-\bd_{k+1}\right)\\
\bq_{k+1}=\bq_k+\delta\left(\bW''\bv_{k+1}-\be_{k+1}\right).
\end{array}
\right.
\end{equation}
Following the idea of the alternating direction method of multipliers (ADMM) \cite{gabay1976dual,bertsekas1989parallel,eckstein1992douglas}, we obtain the following algorithm from \eqref{E:ALM} by minimizing the variables in the first subproblem alternatively:
\begin{equation}\label{E:ADMM}
\left\{\begin{array}{ll}
\bu_{k+1}=\arg\min_{\bu}\{\frac{1}{2}\|\mathbf A\mathbf{u}-\mathbf{f}\|_2^2+\frac{\mu}{2}\|\bW'\bu-\bd_k+\bp_k\|_2^2\}\\
\bv_{k+1}=\arg\min_{\bv}\{\nu_1\|\bd_k-\bv\|_1+\frac{\mu}{2}\|\bW'' \bv-\be_k+\bq_k\|_2^2\}\\
\bd_{k+1}=\arg\min_{\bd}\{\nu_1\|\bd-\bv_{k+1}\|_1+\frac{\mu}{2}\|\bW' \bu_{k+1}-\bd+\bp_k\|_2^2\}\\
\be_{k+1}=\arg\min_{\be}\{\nu_2\|\be\|^q_q+\frac{\mu}{2}\|\bW'' \bv_{k+1}-\be+\bq_k\|_2^2\}\\
\bp_{k+1}=\bp_k+\delta\left(\bW'\bu_{k+1}-\bd_{k+1}\right)\\
\bq_{k+1}=\bq_k+\delta\left(\bW''\bv_{k+1}-\be_{k+1}\right).
\end{array}
\right.
\end{equation}
Note that each of the subproblem of \eqref{E:ADMM} has a closed-form expression. Thus, our proposed algorithm solving the general model \eqref{UnifiedModel:Numerical} is written as
\begin{equation}\label{E:Algorithm}
\left\{\begin{array}{ll}
\bu_{k+1}=\left(\bA^T\bA+\mu \mathbf I\right)^{-1}\left(\bA^T\bf+\mu\bW'^T(\bd_k-\bp_k)\right)\\
\bv_{k+1}=\mathcal T_{\nu_1/\mu}\left(\bW''^T(\be_k-\bq_k)-\bd_k\right)+\bd_k\\
\bd_{k+1}=\mathcal T_{\nu_1/\mu}\left(\bW'\bu_{k+1}+\bp_k-\bv_{k+1}\right)+\bv_{k+1}\\
\be_{k+1}=\left\{\begin{array}{ll}
\mathcal T_{\nu_2/\mu}\left(\bW''\bv_{k+1}+\bq_k\right)\ \ \ (q=1)\\
\frac{\mu}{2\nu_2+\mu}\left(\bW''\bv_{k+1}+\bq_k\right)\ \ \ (q=2)\\
\end{array}
\right.\\
\bp_{k+1}=\bp_k+\delta\left(\bW'\bu_{k+1}-\bd_{k+1}\right)\\
\bq_{k+1}=\bq_k+\delta\left(\bW''\bv_{k+1}-\be_{k+1}\right).
\end{array}
\right.
\end{equation}

Finally, we present results of image deblurring using the general model \eqref{UnifiedModel:Numerical} solved by algorithm \eqref{E:Algorithm}, and compare the results with the analysis based model solved by the split Bregman algorithm/ADMM. For simplicity, we used piecewise linear B-spline framelets given by Example \ref{example:linear} for all the wavelet frame transforms used in the analysis based model and the general model. The blur kernel is a known filter of size $5\times 5$. Mild Gaussian white noise is added to form the observed blurry and noisy image $\mathbf{f}$. All parameters of the models and algorithms are manually chosen to obtain optimal reconstruction results. The original image, observed blurry and noisy image, restored images using the analysis based model \eqref{analysis_discrete} and the general model \eqref{UnifiedModel:Numerical} are presented in Figure \ref{fig_a_u}, where we can see that the general model outperforms the analysis based model as expected.

\begin{figure}
\centering
\includegraphics[width=3.5cm]{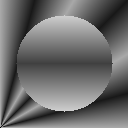}
\includegraphics[width=3.5cm]{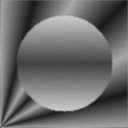}
\includegraphics[width=3.5cm]{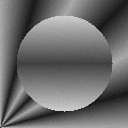}
\includegraphics[width=3.5cm]{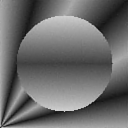}
\caption{From left to right, the original image, the blurry and noisy, and the deblurring images using the analysis based model \eqref{analysis_discrete} with $a=0.2$ and the general model \eqref{UnifiedModel:Numerical} with $q=1$, $\nu_1=0.2$, $\nu_2=0.2$. The PSNR values of the recovered images from the analysis based model and the general model are 37.9372 and 38.5268 respectively.}\label{fig_a_u}
\end{figure}

\bibliographystyle{ieeetr}
\bibliography{ReferenceLibrary}

\end{document}